\documentclass[12pt,psamsfonts]{amsart}

\usepackage{amssymb,amsfonts,amsmath}
\usepackage{enumerate}
\usepackage{mathrsfs}
\usepackage{fullpage}
\usepackage{xspace}
\usepackage[margin=1.0in]{geometry}
\usepackage{tcolorbox}
\usepackage{tikz-cd}
\usepackage{color}
\usepackage{aliascnt}
\usepackage[foot]{amsaddr}
\usepackage{hyperref}
\usepackage{graphicx}
\usepackage[algoruled,linesnumbered]{algorithm2e}

\usepackage{enumitem}
\setlist[enumerate]{label=$(\mathrm{\arabic*})$, leftmargin=*}
\setlist[itemize]{leftmargin=*}

\newaliascnt{theo}{thm}
\newtheorem{theo}[theo]{Theorem}
\aliascntresetthe{theo}

\newaliascnt{cor}{thm}
\newtheorem{cor}[cor]{Corollary}
\aliascntresetthe{cor}

\newaliascnt{prop}{thm}
\newtheorem{prop}[prop]{Proposition}
\aliascntresetthe{prop}

\newaliascnt{lem}{thm}
\newtheorem{lem}[lem]{Lemma}
\aliascntresetthe{lem}

\newaliascnt{conj}{thm}
\newtheorem{conj}[conj]{Conjecture}
\aliascntresetthe{conj}

\newaliascnt{que}{thm}
\newtheorem{que}[que]{Question}
\aliascntresetthe{que}

\newaliascnt{ass}{thm}
\newtheorem{ass}[ass]{Assumption}
\aliascntresetthe{ass}

\newaliascnt{defnot}{thm}

\aliascntresetthe{defnot}

\newaliascnt{princ}{thm}

\aliascntresetthe{princ}

\theoremstyle{remark}
\newaliascnt{rem}{thm}
\newtheorem{rem}[rem]{Remark}
\aliascntresetthe{rem}

\theoremstyle{definition}

\newaliascnt{defn}{thm}
\newtheorem{defn}[defn]{Definition}
\aliascntresetthe{defn}

\newaliascnt{exmp}{thm}
\newtheorem{exmp}[exmp]{Example}
\aliascntresetthe{exmp}

\newaliascnt{notn}{thm}
\newtheorem{notn}[notn]{Notation}
\aliascntresetthe{notn}

\newcommand{\Z}{\mathbb{Z}\xspace}

\newcommand{\C}{\mathbb{C}\xspace}
\newcommand{\F}{\mathbb{F}\xspace}
\newcommand{\Q}{\mathbb{Q}\xspace}
\newcommand{\G}{\mathbb{G}\xspace}
\newcommand{\kk}{\overline{k}\xspace}
\DeclareMathOperator{\Spec}{Spec}
\DeclareMathOperator{\res}{res}

\DeclareMathOperator{\dv}{div}
\DeclareMathOperator{\alb}{alb}
\DeclareMathOperator{\img}{Im}

\DeclareMathOperator{\Cor}{Cor}

\DeclareMathOperator{\Hom}{Hom}
\DeclareMathOperator{\Pic}{Pic}
\DeclareMathOperator{\Gal}{Gal}

\DeclareMathOperator{\Fil}{Fil}

\DeclareMathOperator{\Alb}{Alb}

\DeclareMathOperator{\Br}{Br}
\DeclareMathOperator{\inv}{inv}
\DeclareMathOperator{\nd}{nd}
\DeclareMathOperator{\CH}{CH}

\DeclareMathOperator{\Bl}{Bl}

\DeclareMathOperator{\NS}{NS}
\DeclareMathOperator{\ur}{ur}

\DeclareMathOperator{\Kum}{Kum}
\makeatletter
\let\c@equation\c@thm
\makeatother
\numberwithin{equation}{section}

\newcommand{\red}{\color{red}}

\newcommand{\blue}{\color{blue}}
\newcommand{\black}{\color{black}}

\title{Local and local-to-global Principles for zero-cycles on geometrically Kummer $K3$ surfaces}

\author[*]{Evangelia Gazaki*} \address[*]{\normalfont Department of Mathematics, University of Virginia, 221 Kerchof Hall, 141 Cabell Dr., Charlottesville, VA, 22904, USA. Email: \texttt{eg4va@virginia.edu}}
\author[**]{Jonathan Love**} \address[**]{\normalfont Mathematical Institute, Leiden University, Einsteinweg 55, 2333 CC Leiden, the Netherlands. Email: \texttt{j.r.love@math.leidenuniv.nl}}

\begin{document}

\maketitle
 
\begin{abstract} 
	  Let $X$ be a $K3$ surface over a $p$-adic field $k$ such that 
	for some abelian surface $A$  isogenous to a product of two elliptic curves, there is an isomorphism over the algebraic closure of $k$ between $X$ and the Kummer surface associated to $A$. 
 Under some assumptions on the reduction types of the elliptic curve factors of $A$, we prove 
	that the Chow group $A_0(X)$ of zero-cycles of degree $0$ on $X$ is the direct sum of a divisible group and a finite group. 
	This proves a conjecture of Raskind and Spiess and of Colliot-Th\'{e}l\`{e}ne and it is the first instance for $K3$ surfaces when this conjecture is proved in full. 
	 This class of $K3$'s 
	includes, among others, the diagonal quartic surfaces. In the case of good ordinary reduction we describe many cases when the finite summand of $A_0(X)$ can be completely determined. 
	
	Using these results, we explore a local-to-global conjecture of Colliot-Th\'{e}lene, Sansuc, Kato and Saito  which, roughly speaking, 
	predicts that the Brauer-Manin obstruction is the only obstruction to Weak Approximation for zero-cycles. We give examples of Kummer surfaces over a number field $F$ where the ramified places of good ordinary reduction contribute nontrivially to the Brauer set for zero-cycles of degree $0$ and we describe cases when an unconditional local-to-global principle can be proved, giving the first unconditional evidence for 
	this conjecture  in the case of  
	$K3$ surfaces.

\end{abstract} 
	
\section{Introduction}
For a smooth projective variety $Y$ over a field $k$ we consider the Chow group of zero-cycles $\CH_0(Y)$. This group has a filtration 
\[\CH_0(Y)\supset A_0(Y)\supset T(Y)\supset 0,\] where $A_0(Y):=\CH_0(Y)^{\deg=0}$ is the subgroup of cycles of degree $0$ and $T(Y)=A_0(Y)^{\alb_Y=0}$ is the kernel of the Albanese map 
$\alb_Y:A_0(Y)\rightarrow\Alb_Y(k).$ 

The group $\CH_0(Y)$ is in general extremely hard to compute, even when $Y$ is a surface, due mainly to how little information we have on the structure of the Albanese kernel. When $k$ is a large transcendental field like $\C$ or $\Q_p$ and the variety $Y$ has positive geometric genus, it follows by works of Mumford (\cite{Mumford1968}) and Bloch (\cite{Bloch1975}) that $T(Y)$ is very large; in particular its maximal divisible subgroup is nontrivial. 
The question that is of interest to us in this article is what can be said about the non-divisible part of the Chow group when $k$ is a finite extension of $\Q_p$ or $\Q$. In \cite{Gazaki/Leal2022, Gazaki/Hiranouchi2021, Gazaki2022weak, Gaz22} the first author joint with I. Leal, T. Hiranouchi and A. Koutsianas obtained local and local-to-global results for certain classes of abelian varieties. In this article we extend these results to a special class of $K3$ surfaces, those that are \textit{geometrically of Kummer type}. That is, over the algebraic closure $\kk$ the surface $Y_{\kk}:=Y\otimes_k \kk$ becomes isomorphic to the Kummer surface $\Kum(A)$ associated to an abelian surface $A$ over $\kk$.    

\subsection{Results over $p$-adic fields} For varieties defined over a finite extension $k$ of the $p$-adic field $\Q_p$ 
the motivating question is the following conjecture due to Raskind and Spiess, motivated by earlier considerations of Colliot-Th\'{e}l\`{e}ne. 
 \begin{conj}\label{localconj2} (\cite[1.4(d,e,f)]{Colliot-Thelene1995}, \cite[Conjecture 3.5.4]{Raskind/Spiess2000}) Let $Y$ be a smooth projective and geometrically connected variety over $k$. The kernel $T(Y)$ of the Albanese map of $Y$ is the direct sum of its maximal divisible subgroup and a finite group.  
 \end{conj} 
 
S. Saito and K. Sato (\cite[Theorem 0.3]{Saito/Sato2010}) obtained a breakthrough result showing the weaker property that the subgroup $A_0(Y)$ is the direct sum of a group divisible by any integer $m$ coprime to $p$ and a finite group. Moreover, they showed that the finite piece vanishes in the case of good reduction (\cite[Corollary 0.10]{Saito/Sato2010}). 

The full conjecture is only known in limited cases by the work of Raskind and Spiess (\cite[Theorem 1.1]{Raskind/Spiess2000}) for certain products of curves whose Jacobians have split semi-ordinary reduction and by the first author and I. Leal for products of elliptic curves with at most one having  supersingular reduction (\cite[Theorem 1.2]{Gazaki/Leal2022}).  
In \autoref{section2} we prove this conjecture for large infinite collections of geometrically Kummer $K3$ surfaces, by pushing forward the information from the abelian surface. Note that for a $K3$ surface $Y$ the Albanese variety is zero, and hence $T(Y)=A_0(Y).$ Our first theorem is the following. 

\begin{theo}\label{localtheointro} (cf.~\autoref{padic1}) Let $k$ be a finite extension of the $p$-adic field $\Q_p$. 
Let $X$ be a $K3$ surface over $k$. Suppose there exists a finite extension $K/k$ such that the base change $X_K:=X\otimes_k K$ becomes isomorphic to the Kummer surface $\Kum(A)$ associated to an abelian surface $A$ over $K$. 
\begin{enumerate}
\item[(i)] Suppose that there exists an isogeny $\phi:E_1\times E_2\to A$ from a product of elliptic curves at most one of which has potentially good supersingular reduction. Then \autoref{localconj2} is true for both $A$ and $X$. 
\item[(ii)] Suppose that the abelian surface $A$ is geometrically simple and it has split semi-ordinary reduction.  Then the group $A_0(X)$ is the direct sum of its maximal divisible subgroup and a torsion group. 
\end{enumerate} 
\end{theo} 

See \autoref{almostord} for the definition of split semi-ordinary reduction. 
To our knowledge this is the first piece of evidence for \autoref{localconj2} for $K3$ surfaces. Verifying the conjecture for abelian surfaces isogenous to products of elliptic curves is also new and produces infinitely many new examples of Jacobians of genus $2$ curves for which \autoref{localconj2} is true (see \autoref{abelianexamples}). A weaker result was obtained in \cite{Gaz22}. 

An important special class of geometrically Kummer $K3$ surfaces are the diagonal quartics, which are hypersurfaces in $\mathbb{P}_k^3$ given by an equation of the form 
\[D: x_0^4+a_1x_1^2+a_2x_2^4+a_3x_3^4=0,\]
with $a_1,a_2,a_3\in k^\times$. Mizukami (\cite{Mizukami78}) showed that any such surface becomes  isomorphic over an explicit finite extension to the Kummer surface $\Kum(A)$ of an abelian surface isogenous to a self-product $C\times C$ of a CM elliptic curve (see \autoref{quarticsection}). 

The following Corollary can be thought of as the $p$-analog of \cite[Corollary 0.10]{Saito/Sato2010}. For simplicity we state it here only for Kummer surfaces and diagonal quartics. 
\begin{cor}\label{pdivisibility2} (cf.~\autoref{unramified1}, \autoref{pdivisible}) Let $k$ be a finite \textbf{unramified} extension of $\Q_p$, where $p$ is an odd prime. Let $X=\Kum(E_1\times E_2)$ be the Kummer surface associated to a product of two elliptic curves with good ordinary or almost ordinary reduction (cf. \autoref{almostord}). Then the group $A_0(X)$ is divisible. If moreover $p\equiv 1\mod 4$, then the same is true for the diagonal quartic surface $D$ assuming that $D$ has good reduction. 
\end{cor}
A similar result was obtained in \cite[Theorem 1.4]{Gazaki/Hiranouchi2021} for products of elliptic curves. Such divisibility results are essential for studying Weak Approximation over number fields. 

If we remove the unramified assumption, the result is in general very far from being true. Suppose that $X=\Kum(E_1\times E_2)$ satisfying the reduction hypotheses of \autoref{localtheointro}. We denote by $A_0(X)_{\nd}$ the non-divisible summand of $A_0(X)$, which by our theorem is a finite group. Our methods show that we have a surjection 
\[T(E_1\times E_2)_{\nd}\{2'\}\twoheadrightarrow A_0(X)_{\nd}\{2'\}\] of the subgroups consisting of elements of order coprime to $2$. 
In \autoref{localresults} we describe many cases when this map is an isomorphism. The key ingredient is an isomorphism between the transcendental Brauer groups of $E_1\times E_2$ and $X$ obtained in \cite[Theorem 2.4]{SZ2012} (see \autoref{mainlocal}). Our results show:
\begin{enumerate}
\item[1.] In all reduction cases there is a finite extension $K/k$ over which the group $A_0(X_K)_{\nd}\{2'\}$ is nontrivial,  and it can have order divisible by arbitrarily large powers of $p$. 
\item[2.] In the case of good ordinary reduction  $A_0(X)_{\nd}\{2'\}$ can often be fully computed. 
\end{enumerate}


\subsection{Weak Approximation for zero-cycles} We next focus on local-to-global principles for zero-cycles. We are interested in the following conjecture, which informally states that the Brauer-Manin obstruction is the only obstruction to Weak Approximation for zero-cycles. This is due to Colliot-Th\'{e}l\`{e}ne and Sansuc for geometrically rational varieties and to Kato and Saito for general varieties. 
\begin{conj}\label{localtoglobal1} (\cite[Conjecture A] 
	{Colliot-Thelene/Sansuc1981}, \cite[Section 7]{Kato/Saito1986}, see also \cite[Section 1.1]{Wittenberg2012})  Let $Y$ be a smooth projective and geometrically connected variety over a number field $F$. The following complex is exact
\begin{equation}\label{complex0}\widehat{\CH_0(Y)}\stackrel{\Delta}{\longrightarrow}
\widehat{\CH_{0,\mathbf{A}}(Y)}\rightarrow\Hom(\Br(Y),\Q/\Z).
\end{equation}
\end{conj} 

Here for an abelian group $M$ we denoted by $\widehat{M}=\varprojlim_n M/n$ the profinite completion of $M$. Moreover, if $\Omega$ (resp. $\Omega_f$) is the set of places (resp. finite places) of $F$, then the adelic Chow group $\CH_{0,\mathbf{A}}(Y)$ is essentially $\prod_{v\in\Omega_f}\CH_0(Y_v)$ with a small $2$-torsion contribution from the infinite real places. 

Exactness of \eqref{complex0} would imply exactness of the degree zero piece, 

\begin{equation}\label{complex00}\widehat{A_0(Y)}\stackrel{\Delta}{\longrightarrow}
\widehat{A_{0,\mathbf{A}}(Y)}\rightarrow\Hom(\Br(Y)/\Br(F),\Q/\Z),
\end{equation} and if $Y(F)\neq\emptyset$, the converse is also true. In this article we focus on the complex \eqref{complex00}. 
\begin{defn}\label{Brauerset} We define the Brauer set for  zero-cycles of degree $0$ and denote it $(\widehat{A_{0,\mathbf{A}}(Y)})^{\Br}$ to be the subgroup of the adelic group $\widehat{A_{0,\mathbf{A}}(Y)}$ consisting of elements that are orthogonal to the Brauer group.  
\end{defn}
Suppose that $\Pic(Y\otimes_F\overline{F})$ is finitely generated and torsion-free and that the quotient $\Br(Y)/\Br(F)$ is finite. Two important classes of varieties that satisfy these assumptions are rationally connected varieties and $K3$ surfaces. The former have geometric genus zero, a property which in general implies that both the global and the local Chow groups are small. In fact, it follows by \cite{Kollar/Szabo2003} that $A_0(Y_v)=0$ for every finite place $v$ of good reduction.
 For this class of varieties there are far reaching results regarding the validity of \autoref{localtoglobal1} due to Liang (\cite{Liang2013}) and Harpaz, Wittenberg (\cite{Harpaz/Wittenberg2018}), who used the fibration method to reduce the question to the one for rational points. 
 On the other hand, if $Y$ is a 
$K3$ surface, then $H^2(Y,\mathcal{O}_Y)\neq 0$, which forces all the local Chow groups $\CH_0(Y_v)$ to be enormous. 
Ieronymou (\cite[Theorem 1.3]{Ieronymou2021}) used the fibration method  for $K3$ surfaces to ``approximate modulo $n$", assuming that the Brauer-Manin obstruction is the only obstruction to Weak Approximation for rational points on every finite extension $L/F$ (see \autoref{weaksection} for the precise statement). 
 The following questions arise when $Y$ is a $K3$ surface.

\begin{que}\label{Q1} Does $(\widehat{A_{0,\mathbf{A}}(Y)})^{\Br}$ consist of genuine local zero-cycles? 
\end{que}
Note that if \autoref{localconj2} is true for $Y_v$ for every finite place $v$, then in this case this reads $A_0(Y_v)=D_v\oplus P_v$, where $D_v$ is a divisible group and $P_v$ a finite group. Then $\widehat{A_0(Y_v)}\simeq P_v$, and hence the entire adelic group $\widehat{A_{0,\mathbf{A}}(Y)}$ will consist of genuine zero-cycles.
\begin{que}\label{Q2} Which places contribute nontrivial components in $(\widehat{A_{0,\mathbf{A}}(Y)})^{\Br}$? Are they only finitely many? 
\end{que} 

\autoref{pdivisibility2} suggests that only places of bad reduction and ramified places of good reduction might contribute nontrivially, although there is still an infinite set of good reduction places for which we don't have enough information. 
Such a finiteness expectation is also supported by general expectations about zero-cycles over number fields combined with the finiteness of the Brauer group $\Br(Y)/\Br(F)$ (see \autoref{BB}). 
\begin{que}\label{Q3} Do the places of good reduction contribute to the Brauer set? 
\end{que}

\begin{que}\label{Q4} Can we obtain any unconditional evidence for \autoref{localtoglobal1} for $K3$ surfaces without assuming anything about the arithmetic of rational points? 
\end{que}

In \autoref{weaksection} we give answers to all these questions and obtain evidence for \autoref{localtoglobal1} for geometrically Kummer $K3$ surfaces by considering one prime at a time. Namely, fixing a prime number $p$, one can explore the exactness of the complex 
\begin{equation}\label{complexp0}\varprojlim\limits_{n}A_0(Y)/p^n\stackrel{\Delta}{\longrightarrow}
\prod_{v\in\Omega}\varprojlim\limits_{n}A_0'(Y_v)/p^n\rightarrow\Hom(\Br(Y)\{p\}/\Br_0(Y)\{p\},\Q/\Z). 
\end{equation} 
For the definition of $A_0'(Y_v)$ see \autoref{weaksection}.
Exactness of the complex \eqref{complexp0} for every prime $p$ will imply exactness of \eqref{complex00}. 
Our first theorem is the following.  

\begin{theo}\label{localglobalintro} (cf.~\autoref{localglobal1}) 
Let $X$ be a $K3$ surface over a number field $F$. Suppose there exists a finite extension $K/F$ over which $X$ becomes isomorphic to the Kummer surface associated to an abelian surface $A$ over $K$  admitting an isogeny $E_1\times E_2\xrightarrow{\phi} A$ from a product of elliptic curves. 
 Then there exists an infinite set $T$ of rational primes $p$ for which the following are true:
\begin{enumerate}
\item[(a)]  For every $p\in T$ the middle term of the complex \eqref{complexp0} is the following finite group \[\prod_{v|p}A_0(X_v)_{\nd}\{p\}.\] In particular, it consists of genuine local zero-cycles. 
\item[(b)] Suppose that for every finite extension $L/F$ the set $X(L)$ is dense in the Brauer set $X(\mathbf{A}_L)^{\Br(X_L)}$. Then for every $p\in T$ the complex \eqref{complexp0} is exact.
\item[(c)] For all but finitely many $p\in T$ the middle term of the complex \eqref{complexp0} vanishes and for these primes the density assumption  can be dropped. 
\end{enumerate} 
\end{theo}
Part (b) of the above theorem can be thought of as a ``passing to the limit" version of the result of Ieronymou for geometrically split Kummer $K3$ surfaces (see \autoref{geomKummer}).

In \autoref{goodredplaces} we construct examples of Kummer surfaces for which the ramified places of good ordinary reduction contribute nontrivially to the Brauer set for zero-cycles of degree $0$, which gives an affirmative answer to  \autoref{Q3}. This shows that the behavior for $K3$ surfaces is quite different from the one for rationally connected varieties. 

\begin{prop}\label{Br0} (cf.~\autoref{Brorthogonal}) Let $E$ be an elliptic curve over $\Q$ such that $E\otimes_{\Q}\overline{\Q}$ has complex multiplication by the full ring of integers  of a quadratic imaginary field $K$. Let $A=E\times E$ and $X=\Kum(A)$. Let $p\geq 5$ be a prime that splits completely in $K$ and is coprime to the conductor $\mathfrak{n}$ of $E$. Suppose that the reduction $\overline{E_p}$ of $E$ modulo $p$ satisfies  $|\overline{E_p}(\F_p)|=p$. Then there exists a finite extension $L/K$ of degree $p-1$ totally ramified at the places of $K$ lying above $p$ such that for the Kummer surface $X_L$ the following are true:
\begin{enumerate}
\item[(i)] The Brauer group $\Br(X_L)\{p\}/\Br_0(X_L)\{p\}$ vanishes. 
\item[(ii)] The middle term of the complex \eqref{complexp0} becomes $\displaystyle\prod_{v|p}A_0(X_{L_v})_{\nd}\{p\}=\Z/p$. 
\end{enumerate}
\end{prop}

This proposition follows easily by an analogous result we proved in \cite[Theorem 1.6]{Gazaki2022weak} for the product $A=E\times E$ and by the explicit computations we carry out in \autoref{computations1}. In \autoref{example} we give  examples of pairs $(E,p)$ of elliptic curves $E_{/\Q}$ and primes $p$ satisfying the assumptions of the proposition for each of the quadratic imaginary fields $K=\Q(\sqrt{D})$ with $D=-3, -11, -19, -43, -67, -163$. 

We close this article by addressing \autoref{Q4} in the situation of \autoref{Br0}. The question becomes whether we can unconditionally lift the group $\displaystyle\prod_{v|p}A_0(X_{L_v})_{\nd}\{p\}=\Z/p$ to global zero-cycles. 
 We show that this question can be reduced to the corresponding one for the abelian surface $A=E\times E$, for which in \cite[Theorem 1.7]{Gazaki2022weak} we described sufficient conditions for elliptic curves with positive rank over $\Q$ that guarantee such a lifting. This allows us to find explicit examples of Kummer surfaces satisfying a nontrivial local-to-global principle, in the sense that global zero-cycles need to be constructed. This gives the first \textbf{unconditional evidence} for \autoref{localtoglobal1} for $K3$ surfaces. 
The following example is deduced by a case study carried out in the Appendix of \cite{Gazaki2022weak} written by A. Koutsianas. 
\begin{exmp}\label{weakevidence} Consider the family of elliptic curves given by the Weierstrass equations $\{E_{n}:y^2=x^3-2+7n, n\in\Z\}$ and the prime $p=7$. Note that for each $n\in\Z$ the pair $(E_n,p)$ satisfies the assumptions of \autoref{Br0}. A SAGE computation shows that for about $87\%$ of elliptic curves with rank $1$ in the family $E_n$ with $n\in[-5000, 5000]$ the Kummer surface $X_n=\Kum(E_n\times E_n)$ satisfies unconditionally a local-to-global principle for zero-cycles of degree $0$ after extending to the appropriate degree $6$ extension of $\Q(\zeta_3)$. Many more examples of this flavor have been recently constructed in \cite{Wills25}, where also a heuristic argument is given that explains the percentages that we get. 
\end{exmp}

 


\subsection{Notation}
For a variety $X$ over a field $k$ and an extension $K/k$ we will denote by $X_K:=X\otimes_k K$ the base change to $K$. For a closed point $x\in X$ we will denote by $k(x)$ the residue field of $x$. 
 For an abelian group $B$ and an integer $n$ we will denote by $B_n:=\ker(B\xrightarrow{n}B)$ the $n$-torsion subgroup of $B$ and by $B/n:=B/nB$ the $n$-cotorsion. For a prime $p$ we will denote by $B\{p\}$ the $p$-primary torsion subgroup of $B$. In the special case of an abelian variety $A$ over a field $k$ we will instead use the notation $A[n]$ for the $n$-torsion. For an abelian group $M$ we will denote  by $\widehat{M}:=\varprojlim\limits_{n}M/n$ the profinite completion of $M$. For a homomorphism $f:A\to B$ of abelian groups we will denote by $\widehat{f}:\Hom(B,\Q/\Z)\to\Hom(A,\Q/\Z)$ the dual homomorphism given by precomposition with $f$. 
\subsection{Acknowledgment} The first author's research was partially supported by the NSF grant DMS-2302196. The second author was partially supported by a CRM--ISM postdoctoral fellowship. We are truly grateful to Professor Olivier Wittenberg for sharing some excellent insights with us; \autoref{maxunr} in particular was an outcome of our discussions. We are thankful to Professor Evis Ieronymou for many helpful discussions and to the anonymous referee whose suggestions helped improve the paper. 

\vspace{3pt}
\section{The local Conjecture}\label{section2}
In this section we will prove \autoref{localtheointro}. We start with some general results about zero-cycles on Kummer surfaces. 
For  most of this section unless otherwise mentioned $k$ will be an arbitrary base field. 
\subsection{Background and notation} 
Let $Y$ be a smooth projective variety over $k$. The Chow group $\CH_0(Y)$ of zero-cycles has a filtration \[\CH_0(Y)\supset A_0(Y)\supset T(Y)\supset 0,\]
where $A_0(Y)=\CH_0(Y)^{\deg=0}$ is the kernel of the degree map, $\CH_0(Y)\xrightarrow{\deg}\Z$, and $T(Y)=A_0(Y)^{\alb_Y=0}$ is the kernel of the Albanese map, $A_0(Y)\xrightarrow{\alb_Y} \Alb_Y(k)$.
When $Y$ is a surface, it follows by \cite[Proposition 6.3]{CT79} that the group $A_0(Y)$ is a birational invariant. This is a key ingredient in obtaining results for Kummer surfaces. 
\begin{exmp} Let $Y$ be a $K3$ surface over $k$. Since the Albanese variety is dual to the Picard variety $\Pic^0(Y)$ and $Y$ is simply connected, it follows that $\Alb_Y=0$, and hence $A_0(Y)=T(Y)$ (see for example \cite[p.~241]{Huy16}).  
\end{exmp}
\begin{exmp}\label{abexmp} Let $A$ be an abelian variety over $k$. Then $\Alb_A=A$ and the Albanese map has the following simple description 
\begin{eqnarray*}
\alb_A: && A_0(A)\to A(k)\\
&& \sum_{x\in A}n_x[x]\mapsto\sum_{x\in A}n_x N_{k(x)/k}(x),
\end{eqnarray*}
where $N_{k(x)/k}:A(k(x))\to A(k)$ is the usual norm map on abelian varieties.
Here we are abusing notation: for a closed point $x\in A$, we also denote by $x$ the induced $k(x)$-rational point of $A$.
\end{exmp}

\subsection{Kummer surfaces}\label{Kummer1}
Let $A$ be an abelian surface over $k$. Let $A[2]$ be its subgroup of $2$-torsion elements. Let $A':=\Bl_{A[2]}A$ be the surface obtained by $A$ by blowing-up the subscheme $A[2]$  and let $\rho:A'\to A$ be the corresponding birational morphism. The preimage of $A[2]$ is a $1$-dimensional scheme $V\subseteq A'$ that splits over $\bar{k}$ into $16$ rational curves (one lying over each $\bar{k}$-point of $A[2]$), and $\rho$ induces an isomorphism \[\rho:A'- V\xrightarrow{\simeq} A- A[2].\]\black
Let $X_0:=A/\langle -1\rangle$ be the quotient of $A$ by the negation involution. We have a degree two cover $\pi:A\to X_0$. The singular locus of $X_0$ is $\pi(A[2])$, which splits into sixteen singularities over $\bar{k}$. Let $X$ be the surface obtained from $X_0$ by blowing up the singular locus. The variety $X:=\Kum(A)$ is a $K3$ surface that is known as \textit{the Kummer surface} associated to $A$. We will denote by $W$ the preimage of $\pi(A[2])$ in $X$. It follows by the universal property of blow-ups that $\pi$ induces a degree $2$ cover $\pi':A'\to X$ with ramification locus $V$, sending $V$ isomorphically onto $W$ (see \cite[p.~118]{SZ2012}).

We consider the Chow group of zero-cycles $\CH_0(X)$. 
The proper morphism $\pi:A'\to X$ induces a push-forward homomorphism \[\pi_\star: A_0(A')\to A_0(X).\] 
Moreover, the birational morphism $\rho$ induces an isomorphism (cf.~\cite[Proposition 6.3]{CT79}) 
\[\rho_\star:A_0(A)\xrightarrow{\simeq} A_0(A').\] 
Since $\Alb_{A'}=\Alb_A=A$, we have a commutative diagram
\[\begin{tikzcd}
0\ar[r] & T(A')\ar[r] & A_0(A')\ar{r}{\alb_{A'}}\ar{d}{\rho_\star} & A(k)\ar{r}\ar{d} & 0\\
0\ar[r] & T(A)\ar[r] & A_0(A)\ar{r}{\alb_{A}} & A(k)\ar{r} & 0
\end{tikzcd}\] with exact rows and the two vertical maps isomorphisms. Thus, by a diagram chase and the five lemma we obtain an isomorphism $\rho_\star:T(A')\xrightarrow{\simeq}T(A)$.

	\begin{prop}\label{quotients} Let $A$ be an abelian surface over  $k$ and $X=\Kum(A)$ its associated Kummer surface. The quotient $A_0(X)/\pi_\star(T(A'))$ is $2$-torsion. 
	\end{prop} 
	\begin{proof}
		We have a pullback morphism $\pi^\star:A_0(X)\to A_0(A')$, and $\pi_\star\circ\pi^\star$ equals multiplication by $2$ on $A_0(X)$. By functoriality of albanese maps with respect to pullback (see for example \cite[Proof of Theorem 5.2]{GR25}), there is a commutative square
		\[\begin{tikzcd}
			A_0(X)\arrow[r,"\alb_X"]\arrow[d,"\pi^\star"] & 0\arrow{d} \\
			A_0(A')\arrow[r,"\alb_{A'}"] & A(k).
		\end{tikzcd}\]
		Thus the image of $A_0(X)$ under $\pi^\star$ is contained in $T(A')$, so the image of $A_0(X)$ under multiplication by $2$ is contained in $\pi_\star(T(A'))$.
		
	\end{proof}

From now on we will abuse notation and identify $A_0(A)$ with $A_0(A')$ under the isomorphism $\rho_\star$, so we have $\pi_\star:A_0(A)\to A_0(X)$ and $2A_0(X)/\pi_\star(T(A))=0$. \autoref{quotients} yields the following Corollary. 

\begin{cor}\label{coprimeto2} 
	 Let $A$ be an abelian surface over $k$ and $X=\Kum(A)$ its associated Kummer surface. 
	Let $n\geq 1$ an odd integer. Then $\pi_\star$ induces a surjection $\pi_\star:T(A)/n\twoheadrightarrow A_0(X)/n$. 
\end{cor}

\vspace{2pt}

\subsection{Proof of \autoref{localtheointro}}\label{localconjsection}
We now focus on \autoref{localconj2}. Throughout this subsection $k$ will be a finite extension of the $p$-adic field $\Q_p$ for a prime $p$. We will denote by $\kappa$ the residue field of $k$. 

\begin{notn} For an abelian group $B$ we will denote by $B_{\dv}$ its maximal divisible subgroup
 and by $B_{\nd}$ the non-divisible quotient $B/B_{\dv}$. Since divisible groups are injective $\Z$-modules, we have a decomposition $B\simeq B_{\dv}\oplus B_{\nd}$. Thus, we can view $B_{\nd}$ as a subgroup of $B$. 
 \end{notn} Using the above notation we can restate  \autoref{localconj2} as follows.
 
 \begin{conj}\label{conjnd} For a smooth projective variety $Y$ over $k$ the group $T(Y)_{\nd}$ is finite. 
 \end{conj} 
 When $Y$ is a $K3$ surface this amounts to proving that the group $A_0(Y)_{\nd}$ is finite. 
 Before proving the main theorem of this section, we recall the following definitions about reduction types of abelian surfaces over $p$-adic fields (see \cite[p.~1]{LazdaSkorobogatov}, \cite[p.~14]{Raskind/Spiess2000})).
\begin{defn}\label{almostord} Let $A$ be an abelian surface over $k$ and $\mathcal{A}$ be its N\'{e}ron model over the ring of integers $\mathcal{O}_k$ of $k$. 
\begin{enumerate}
\item We say that $A$ has split semi-ordinary reduction if the connected component $(\mathcal{A}\otimes_{\mathcal{O}_k} \kappa)^\circ$ of the special fiber of $\mathcal{A}$ is an extension of an ordinary abelian variety by a split torus. 
\item Suppose that $A$ has good reduction. Let $\overline{A}$ be the reduction of $A$ over $\kappa$. We say that $A$ has good ordinary reduction if $\overline{A}[p](\overline{\kappa})\simeq(\Z/p\Z)^2$, and $A$ has good almost ordinary reduction if $\overline{A}[p](\overline{\kappa})\simeq\Z/p\Z$. 
\end{enumerate} 
\end{defn}
In the special case when $A=E_1\times E_2$ is a product of elliptic curves with good reduction, $A$ has good ordinary or almost ordinary reduction if and only if at most one of the elliptic curves has good supersingular reduction.  

\begin{theo}\label{padic1} 
	Let $k$ be a finite extension of the $p$-adic field $\Q_p$. 
	Let $X$ be a $K3$ surface over $k$. Suppose there exists a finite extension $K/k$ such that the base change $X_K$ becomes isomorphic to the Kummer surface $\Kum(A)$ associated to an abelian surface $A$ over $K$. 
\begin{enumerate}
	\item[(i)] Suppose that there exists an isogeny $\phi:E_1\times E_2\to A$ from a product of elliptic curves at most one of which has potentially good supersingular reduction. Then \autoref{conjnd} is true for both $A$ and $X$. 
	\item[(ii)] Suppose that the abelian variety $A$ is geometrically simple and it has split semi-ordinary reduction. Then the group $A_0(X)_{\nd}$ is torsion.  
\end{enumerate} 
\end{theo}
\begin{proof} 
	We first prove (i).  We start with some preliminary reductions. It follows by \cite[Theorem 8.1]{CT93} that for a smooth projective surface $Y$ over a $p$-adic field and for every positive integer $m$ the group $\CH_0(Y)_m$ is finite. Thus, if $Y$ is a surface, to verify \autoref{conjnd} it is enough to show that the group $T(Y)_{\nd}$ is of bounded torsion. Combining this with \cite[Lemma 3.3]{Gaz22},\footnote{This source assumes $p>2$, but this assumption is not used in the proof of the lemma in question.} to prove finiteness of $T(A)_{\nd}, A_0(X)_{\nd}$ we are allowed to extend to a finite extension $L/k$.  Thus, we may assume we have the following behavior over $k$:  
	\begin{itemize}
	\item $X\simeq\Kum(A)$ is isomorphic over $k$ to the Kummer surface associated to an abelian surface $A$ and there exists an isogeny $\phi:E_1\times E_2\to A$ from a product of two elliptic curves. 
	\item The elliptic curves $E_1,E_2$ have split semistable reduction and at most one of them has good supersingular reduction. 
	\end{itemize}
	Under the second assumption, it follows by \cite[Theorem 1.2]{Gazaki/Leal2022} that the group $T(E_1\times E_2)_{\nd}$ is finite. We show first that the group $T(A)_{\nd}$ is finite. The isogeny $\phi$ is a proper and flat map and thus it induces pushforward $\phi_\star: \CH_0(E_1\times E_2)\to\CH_0(A)$ and pullback $\phi^\star:\CH_0(A)\to\CH_0(E_1\times E_2)$ satisfying $\phi_\star\phi^\star=\deg(\phi)$. These restrict to homomorphisms $\phi_\star: T(E_1\times E_2)\to T(A)$ and $\phi^\star:T(A)\to T(E_1\times E_2)$ (see \cite[Proof of Theorem 5.2]{GR25}). Similarly to what we did in \autoref{Kummer1} we consider the subgroup $\phi_\star(T(E_1\times E_2))$ of $T(A)$. This contains the subgroup $\phi_\star(\phi^\star(T(A)))=\deg(\phi)T(A)$, and hence the quotient $T(A)/\phi_\star(T(E_1\times E_2))$ is $\deg(\phi)$-torsion.

	Since images of divisible groups are divisible, the group $D=\pi_\star(T(E_1\times E_2)_{\dv})$ is a divisible subgroup of $T(A)$. Consider the short exact sequence
	\[0\to D\to T(A)\to T(A)/D\to 0,\] which splits since $D$ is divisible. Set $N=|T(E_1\times E_2)_{\nd}|$ and let $z\in T(A)$. It follows that $\deg(\phi)z\in \phi_\star(T(E_1\times E_2))$ and $N\deg(\phi)z\in D$. Thus, the quotient $T(A)/D$ is of bounded torsion. We conclude that $D$ is the maximal divisible subgroup of $T(A)$ and the group $T(A)_{\nd}=T(A)/D$ is finite. 
	
	To show finiteness of $A_0(X)_{\nd}$ we argue exactly in the same way using the filtration \[\CH_0(X)\supset A_0(X)\supset\pi_\star(T(A))\supset 0.\] In this case we obtain that the quotient $A_0(X)/D$ is $2N\deg(\phi)$-torsion. 
	
	The proof of (ii) is similar. Namely we may assume that $X\simeq\Kum(A)$ is isomorphic over $k$ to the Kummer surface associated to an absolutely simple abelian surface with split semi-ordinary  reduction. Under these assumptions it follows by \cite[Theorem 1.2]{Gazaki2019}\footnote{This theorem assumes $p>2$ and $A$ ordinary, but the result is still true if $p=2$ or if $A$ is split semi-ordinary. In the proof of \cite[Lemma 2.4]{Gazaki2019}, one can note that even if the map $\Phi_2:F^2/F^3\hookrightarrow S_2(k;A)$ fails to be surjective, its image still contains $2S_2(k;A)$ even under these weaker conditions. Therefore $\img\Phi_2$ contains the maximal divisible subgroup of $S_2(k;A)$ as a finite index subgroup, so $\Phi_2$ still establishes an isomorphism between $F^2/F^3$ and the direct sum of a finite group and a divisible group. The conclusion of \cite[Lemma 2.4]{Gazaki2019} still holds, and the rest of the proof of \cite[Theorem 1.2]{Gazaki2019} can proceed as written.} that the group $T(A)_{\nd}$ is torsion. This implies that the group $\Fil^2(X)=\pi_\star(T(A))$ has a decomposition \[\Fil^2(X)=\pi_\star(T(A)_{\dv})\oplus \pi_\star(T(A)_{\nd})=D\oplus \pi_\star(T(A)_{\nd})\] into a divisible group and a torsion group. Then for every $z\in A_0(X)$ there exists some integer $n\geq 1$ such that $2nz\in D$ from where the claim follows. 

\end{proof}

\begin{rem}\label{abelianexamples} We note that verifying \autoref{conjnd} at this level of generality for an abelian surface $A$ isogenous to a product of elliptic curves is new. However, some special cases and weaker results have been previously obtained in \cite[Proposition 1.7]{Gaz22} and \cite[Corollary 1.5]{Gazaki/Hiranouchi2021}. \autoref{padic1} gives infinitely many new examples of abelian surfaces that satisfy \autoref{localconj2}. Namely, in \cite[Theorem 1.3]{gazakilove2023hyperelliptic} given a product $E_1\times E_2$ of two elliptic curves over an algebraic number field $F$  we constructed infinitely many genus $2$ curves $H$ defined over the algebraic closure $\overline{F}$ whose Jacobian is isogenous to $E_{1\overline{F}}\times E_{1\overline{F}}$. Since \autoref{padic1} allows geometric behavior and isogenies of arbitrary degree, we see that for any such Jacobian we can find infinitely many places of its field of definition such that the corresponding local abelian surface satisfies the reduction assumptions of \autoref{padic1}. We refer to \cite[Section 3.5.1]{gazakilove2023hyperelliptic} for explicit Weierstrass equations of genus $2$ curves having Jacobians isogenous to products of elliptic curves. There are also some older constructions over finite extensions of $\Q_p$ by Frey and Kani (\cite{FreyKani}).
\end{rem} 

In the case $X=\Kum(A)$ is itself the Kummer surface associated to a product $A=E_1\times E_2$ of elliptic curves, the proof of \autoref{padic1} yields the following more explicit result. 
\begin{cor}\label{explicit1} Let $X=\Kum(E_1\times E_2)$ be the Kummer surface associated to a product of two elliptic curves over the $p$-adic field $k$. Assume that both elliptic curves have split semistable reduction and at most one of them has good supersingular reduction. Then $A_0(X)_{\nd}$ is a finite group of exponent divisible by $2|T(E_1\times E_2)_{\nd}|$.
\end{cor}
\begin{defn}\label{geomKummer} Let $X$ be a $K3$ surface over a field $k$. We say that $X$ is \textit{geometrically split Kummer} if the base change $X_{\kk}$ to the algebraic closure is isomorphic to the Kummer surface $\Kum(A)$ associated to an abelian surface $A$ admitting an isogeny $\phi:E_1\times E_2\to A$ from a product of elliptic curves.  
\end{defn}
\begin{rem} We see that \autoref{padic1} gives very large collections of geometrically split Kummer $K3$ surfaces that satisfy Colliot-Th\'{e}l\`{e}ne's conjecture. Even in the non-split case we obtain some weaker evidence. 
For any such $K3$ surface $X$ over $k$, the base change $X_{\kk}$ has very large Picard rank. It will be interesting to find evidence for \autoref{localconj2} for $K3$ surfaces with small Picard rank. 
\end{rem}

\vspace{1pt}
\subsection{Diagonal Quartic Surfaces}\label{quarticsection} In this section we obtain a result similar to \autoref{explicit1} for a special class of geometrically split Kummer $K3$ surfaces, the diagonal quartics. 
 A recent reference for the facts mentioned below is given in the Appendix of \cite{ISZ2011} written by Sir Peter  Swinnerton-Dyer. 

Let $k$ be a field of characteristic different from $2$. A diagonal quartic surface over $k$ is a smooth hypersurface $D$ in $\mathbb{P}_k^3$ given by the equation 
\begin{equation}\label{quartic}
x_0^4+a_1x_1^2+a_2x_2^4+a_3x_3^4=0, 
\end{equation} where $a_1,a_2,a_3\in k^\times$. A famous special case is the \textit{Fermat quartic} surface $Y$,   
\begin{equation}
x_0^4+x_1^4+x_2^4+x_3^4=0.
\end{equation}
Consider the finite extension $K=k(\mu_8)$. Mizukami in his thesis (cf.~\cite{Mizukami78}) constructed an isomorphism
$f:Y_K\xrightarrow{\simeq} \Kum(A_K)$ defined over $K$, where $A$ is the abelian surface obtained as follows. Let $C$ be the elliptic curve  which is a smooth projective model of the affine curve with equation $y^2=(x^2-1)(x^2-\frac{1}{2})$. The base point of $C$ is the point at infinity at which $y/x^2=1$. Let $\tau: C\to C$ be the fixed point free involution on $C$ given by $\tau(x,y)=(-x,-y)$. The abelian surface $A$ is the quotient of $C\times C$ modulo the simultaneous action of $\tau$ on both components. 

We see that there is a degree $2$ isogeny $C\times C\to A$, making the Fermat quartic geometrically split Kummer. 
The more general diagonal quartic $D$ given by \eqref{quartic} becomes isomorphic to $Y$ over the extension $K(\sqrt{a_1},\sqrt{a_2},\sqrt{a_3})$. We conclude that in the most general case there is a finite extension $L/k$ of degree $2^r$ for some $r\leq 5$ such that $D_L\simeq \Kum(A_L)$. 
We obtain the following explicit result. 
\begin{cor}\label{explicit2} Let $p$ be a prime such that $p\equiv 1\mod 4$ and $p$ is coprime to the conductor $\mathfrak{n}$ of $C$. Let $k$ be a finite extension of $\Q_p$ and $D$ be the diagonal quartic surface \eqref{quartic} over $k$. Then $A_0(D)_{\nd}$ is a finite group of exponent divisible by $2^{r+1}|T(C_L\times C_L)_{\nd}|$.  
\end{cor}
\begin{proof} We first claim that every prime $p$ satisfying the assumptions is a prime of good ordinary reduction for the elliptic curve $C$. To see this note that $C$ is isogenous to the elliptic curve $E:=C/\tau$ which has a short Weierstrass equation $y^2=x^3-4x$. 
Thus, $E$ has CM by $\Z[i]$, and hence it has good ordinary reduction at every prime $p$ that splits completely in $\Z[i]$ and is coprime to $\mathfrak{n}$; equivalently at every prime $p\equiv 1\mod 4$ and coprime to $\mathfrak{n}$. Since the isogeny $\phi: C\to E$ is of degree coprime to $p$, $C$ has also good ordinary reduction. 

We conclude that for the surface $D$ the assumptions of \autoref{padic1} (i) apply. 
The proof of this theorem gives a vanishing $(2|T(C_L\times C_L)_{\nd}|)A_0(D_L)_{\nd}=0$. 
Combining this with the proof of \cite[Lemma 3.3]{Gaz22} we obtain the desired vanishing \[2^{r+1}|T(C_L\times C_L)_{\nd}|A_0(D)_{\nd}=0.\]

\end{proof}

\vspace{2pt}
\section{Explicit local structure}\label{localresults} 
Throughout this section $k$ will be a finite extension of $\Q_p$ where $p$ is a prime; we will often assume that $p$ is odd. 
\black
Let $X$ be a $K3$ surface over $k$ satisfying the assumptions of \autoref{padic1} (i). 
In this section we obtain more concrete information on the non-divisible part $A_0(X)_{\nd}$ of the Albanese kernel, which in certain cases can be fully computed. 
\subsection{Unramified Divisibility} We start with a result over unramified extensions of $\Q_p$. 

%
\begin{cor}\label{unramified1} 
Let $p>2$ and  $k$ be a finite \textbf{unramified} extension of $\Q_p$. 
\begin{enumerate}
\item[(a)] Let $X=\Kum(A)$ be the Kummer surface associated to a product $A=E_1\times E_2$ of two elliptic curves over $k$.  Suppose that $A$ has good ordinary or almost ordinary reduction. Then the group $A_0(X)$ is divisible. 
\item[(b)] Let $D$ be a diagonal quartic surface \eqref{quartic} over $k$ and $p$ a prime satisfying the assumptions of \autoref{explicit2}. Then the finite group $A_0(D)_{\nd}$ is $2$-primary torsion. If additionally $D$ has good reduction, then $A_0(D)$ is divisible.
\end{enumerate}  
\end{cor}
\begin{proof} 
It follows by \cite[Theorem 1.4]{Gazaki/Hiranouchi2021} that if $E_1, E_2$ are elliptic curves with good reduction over an unramified  extension of $\Q_p$  for odd $p$, and  at least one has good ordinary reduction, then the group $T(E_1\times E_2)_{\nd}$ vanishes. Combining this with Corollaries \eqref{explicit1}, \eqref{explicit2} we obtain that the finite groups $T(X)_{\nd}$ and $T(D)_{\nd}$ are $2$-primary torsion. 

Next, it follows by \cite[Lemma 4.2]{Matsumoto2015} that if the abelian surface $A$ has good reduction, then so does its associated Kummer surface $X$. It then follows by \cite[Proposition 3.4]{Ieronymou2021}, \cite[Corollary 0.10]{Saito/Sato2010} that the group $A_0(X)$ is $l$-divisible for every prime $l\neq p$. Since we assumed that $p$ is odd, it follows that $A_0(X)$ is divisible. Lastly, the same will be true for the diagonal quartic $D$ if we assume that $D$ has good reduction.

\end{proof} 

\begin{rem}\label{ordreduction} If $A$ is an abelian surface over a $p$-adic field with good ordinary reduction it follows by \cite[Remark 6.5]{LazdaSkorobogatov} that the associated Kummer surface $X=\Kum(A)$ has also ordinary reduction in the sense of Bloch and Kato (\cite{BlochKato86}). 
\end{rem}

\begin{rem}\label{pdivisible} \autoref{unramified1} can be extended to more general geometrically split Kummer $K3$ surfaces $X$ over an unramified extension $k$ of $\Q_p$ for $p>2$. Namely, suppose there is a finite unramified extension $K/k$ over which there is an isomorphism $X_K\simeq\Kum(A)$ for some abelian surface $A$ over $K$. Suppose $E_1\times E_2\xrightarrow{\phi}A$ is an isogeny and that the residue characteristic $p$ is coprime to $2[K:k]\deg(\phi)$. Since $p\nmid 2\deg(\phi)$, we obtain a surjection 
\[\pi_\star\phi_\star:T(E_1\times E_2)/p\twoheadrightarrow A_0(X_K)/p.\]
Let $N_{K/k}: \CH_0(X_K)\to \CH_0(X),$ $ \res_{K/k}:\CH_0(X)\to\CH_0(X_K)$ be respectively the pushforward and pullback maps induced by the projection $\pi_{K/k}: X_K\to X$. Since the composition $N_{K/k}\res_{K/k}=[K:k]$, which is coprime to $p$, it follows that the induced homomorphism $A_0(X_K)/p\xrightarrow{N_{K/k}}A_0(X)/p$ of $\F_p$-vector spaces is surjective. 
 If we additionally assume that the abelian surface $E_1\times E_2$ has good ordinary or almost ordinary reduction, then once again we can conclude that the group $A_0(X)$ is $p$-divisible by composing the two surjections. 
\end{rem}

As we shall see in \autoref{weaksection}, the main significance of \autoref{unramified1} and \autoref{pdivisible} is for $K3$ surfaces defined over number fields. In particular, this is a key component to obtain results relevant to Weak Approximation for zero-cycles. 



\vspace{2pt}
\subsection{Computing the non-divisible subgroup and relation to the Brauer group}\label{computations1}
 For the remainder of this section we consider the following set-up.
\begin{ass}\label{setup}  Let $k$ be a finite extension of $\Q_p$. Let $E_1, E_2$ be elliptic curves over $k$ with at most one having potentially good supersingular reduction. Set $A=E_1\times E_2$ and $X=\Kum(A)$.  
\end{ass}


\begin{notn} For a finite abelian group $B$ we will denote by $B\{2'\}$ the subgroup of $B$ consisting of all elements of order coprime to $2$. 
\end{notn} 
\autoref{padic1} gives that the group $A_0(X)_{\nd}$ is finite of order divisible by $N:=|T(A)_{\nd}\{2'\}|$. Since $N$ is coprime to $2$, \autoref{coprimeto2} gives a surjection 
\begin{equation}\label{mainsurjection}
\pi_\star: T(A)_{\nd}\{2'\}\twoheadrightarrow A_0(X)_{\nd}\{2'\}. 
\end{equation}

%


 The purpose of this subsection is to describe cases where the group $A_0(X)_{\nd}\{2'\}$ is nontrivial and the above surjection is an isomorphism.


 For a smooth projective variety $Y$ of dimension $d$ over a $p$-adic field $k$ the non-divisible subgroup of the Chow group $\CH_0(Y)$ of zero-cycles can often be understood by relating it to certain \'{e}tale cohomology groups. More precisely, for an integer $n\geq 1$ we have:
 \begin{enumerate}
 \item[(a)] The cycle class map to \'{e}tale cohomology 
\[c_n^Y:\CH_0(Y)/n\to H^{2d}_{\text{\'{e}t}}(Y,\mu_{n}^{\otimes d}).\]
\item[(b)] The Brauer-Manin pairing 
\[\langle,\rangle_Y:\CH_0(Y)\times\Br(Y)\to\Q/\Z,\] where $\Br(Y)=H^2_{\text{\'{e}t}}(Y,\G_m)$ is the cohomological Brauer group of $Y$. This pairing induces a homomorphism $\varepsilon_n^Y:\CH_0(Y)/n\to\Hom(\Br(Y)_n,\Q/\Z)$. 
 \end{enumerate} 
In what follows we won't need the precise definition of the cycle class map, but we will make use of the Brauer-Manin pairing, whose definition we briefly recall here. For a closed point $y\in Y$ we get a pullback homomorphism $\Br(Y)\xrightarrow{\iota_{y}^\star}\Br(k(y))$ induced by the closed embedding $\iota_y:\Spec(k(y))\hookrightarrow Y$. Then for $\alpha\in\Br(Y)$ we define \[\langle [y],\alpha\rangle_Y=\Cor_{k(y)/k}(\iota_y^\star(\alpha))\in\Br(k),\] where $\Cor_{k(y)/k}:\Br(k(y))\to\Br(k)$ is the Corestriction map of Galois cohomology. Postcomposing with the invariant map of local Class Field Theory, $\inv:\Br(k)\xrightarrow{\simeq}\Q/\Z$, gives the required pairing. For a proof of the fact that the pairing factors through rational equivalence we refer to \cite[p.~53]{Colliot-Thelene1995}. 

It is well-known that the maps $c_n^Y$ and $\varepsilon_n^Y$ are related and that $c_n^Y$ is injective if and only if $\varepsilon_n^Y$ is. We briefly review the argument here. For more details we refer to \cite[p.~41]{CT93} (see also \cite[Remark 3.2]{Esnault/Wittenberg}).
 Using local Class Field Theory and local Tate duality S. Saito (\cite{Saito89})  gave  
 a non-degenerate pairing \[H^{2d}_{\text{\'{e}t}}(Y,\mu_{n}^{\otimes d})\times H^2_{\text{\'{e}t}}(Y,\mu_n)\to\Z/n\Z,\] inducing an injection  $H^{2d}_{\text{\'{e}t}}(Y,\mu_{n}^{\otimes d})\hookrightarrow\Hom(H^2_{\text{\'{e}t}}(Y,\mu_n),\Z/n)=\Hom(H^2_{\text{\'{e}t}}(Y,\mu_n),\Q/\Z)$. Precomposing with the cycle class map we obtain a homomorphism 
\[\CH_0(Y)/n\to \Hom(H^2_{\text{\'{e}t}}(Y,\mu_n),\Q/\Z).\]
The Kummer sequence $1\to\mu_n\to\G_m\xrightarrow{n}\G_m\to 1$ on $Y_{\text{\'{e}t}}$ induces a short exact sequence 
\[0\to\Pic(Y)/n\to H^2(Y,\mu_n)\to\Br(Y)_n\to 0.\] Applying the exact functor $\Hom(-,\Q/\Z)$ gives an exact sequence 
\[0\to\Hom(\Br(Y)_n,\Q/Z)\to\Hom(H^2_{\text{\'{e}t}}(Y,\mu_n),\Q/Z)\to\Hom(\Pic(Y)/n,\Q/Z)\to 0.\]
The pairing $\CH_0(Y)/n\times\Pic(Y)/n\to\Q/\Z$ induced by the cycle class map and the Kummer sequence is zero, which implies that the homomorphism $\CH_0(Y)/n\to \Hom(H^2_{\text{\'{e}t}}(Y,\mu_n),\Q/\Z)$ factors through $\Hom(\Br(Y)_n,\Q/Z)$. The induced homomorphism \[\CH_0(Y)/n\to\Hom(\Br(Y)_n,\Q/Z)\] coincides with the map $\varepsilon_n^Y$ induced by the Brauer-Manin pairing. It then follows by the non-degeneracy of the Saito pairing that $c_n^Y$ is injective if and only if $\varepsilon_n^Y$ is. 

We note that the Brauer group of a variety is contravariant  
functorial. For a morphism $f:Y\to Z$ of smooth projective varieties we will denote by $f^\star:\Br(Z)\to\Br(Y)$ the induced pullback map and by $\widehat{f^\star}$ the dual homomorphism, $\Hom(\Br(Y),\Q/\Z)\xrightarrow{\widehat{f^\star}}\Hom(\Br(Z),\Q/\Z)$. The following lemma about functoriality of the Brauer-Manin pairing is standard. We include a proof for completeness.  

\begin{lem}\label{functoriality} Let $f:Y\to Z$ be a proper morphism between smooth projective varieties over $k$. For every $n\geq 1$ there is a commutative diagram 
\begin{equation}\label{diag1}
\begin{tikzcd}
\CH_0(Y)/n\ar{r}{\varepsilon_n^Y}\ar{d}{f_\star} & \Hom(\Br(Y)_n,\Q/\Z)\ar{d}{\widehat{f^\star}}\\
 \CH_0(Z)/n\ar{r}{\varepsilon_n^Z} & \Hom(\Br(Z)_n,\Q/\Z),
\end{tikzcd}
\end{equation} where the maps $\varepsilon_n^Y, \varepsilon_n^Z$ are induced by the Brauer-Manin pairing. 
\end{lem}
\begin{proof}
The lemma will follow if we show 
\[\langle f_\star([y]),\alpha\rangle_Z=\langle [y],f^\star(\alpha)\rangle_Y,\]
for every closed point $y\in Y$ and all $\alpha\in\Br(Z)$. The right hand side equals 
\[\langle [y],f^\star(\alpha)\rangle_Y=\Cor_{k(y)/k}(\iota_y^\star(f^\star(\alpha))=\Cor_{k(f(y))/k}(\Cor_{k(y)/k(f(y))}((f\circ\iota_y)^\star(\alpha))).\] The left hand side equals 
\[\langle f_\star([y]),\alpha\rangle_Z=\Cor_{k(f(y))/k}([k(y):k(f(y))]\iota^\star_{f(y)}(\alpha)).\] Thus, it suffices to show $\Cor_{k(y)/k(f(y))}((f\circ\iota_y)^\star(\alpha))=[k(y):k(f(y))]\iota^\star_{f(y)}(\alpha)$. This follows immediately by the factoring $f\circ\iota_y=\pi_{k(y)/k(f(y))}\circ\iota_{f(y)}$, where $\pi_{k(y)/k(f(y))}:\Spec(k(y))\to\Spec(k(f(y))$ is the projection. Namely, the pullback $\pi_{k(y)/k(f(y))}^\star:\Br(k(f(y)))\to\Br(k(y))$ coincides with the restriction map $\res_{k(y)/k(f(y))}$ of Galois cohomology and we have an equality $\Cor_{k(y)/k(f(y))}\circ\res_{k(y)/k(f(y))}=\cdot [k(y):k(f(y))]$.

\end{proof}
For a smooth projective variety $Y$ over $k$ the Hochshild-Serre spectral sequence gives a filtration of the Brauer group \[\Br(Y)\supseteq\Br_1(Y)\supseteq\Br_0(Y)\supset 0,\] where $\Br_0(Y):=\img(\Br(k)\to\Br(Y))$ are the constants and $\Br_1(Y)=\ker(\Br(Y)\to\Br(Y_{\kk}))$ is the \textit{algebraic Brauer group}. The quotient $\Br(Y)/\Br_1(Y)$ is known as the \textit{transcendental Brauer group} of $Y$. It follows easily that the Brauer-Manin pairing induces a homomorphism 
\[A_0(Y)/n\xrightarrow{\varepsilon_n^Y}\Hom(\Br(Y)_n/\Br_0(Y)_n,\Q/\Z).\]
 
\begin{prop}\label{mainlocal} Consider the set-up \eqref{setup}.  Suppose that the N\'{e}ron-Severi group $\NS(A_{\kk})$ has trivial $\Gal(\kk/k)$-action. Then for every integer $n\geq 1$ there is a commutative diagram induced by the Brauer-Manin pairing
\begin{equation}\label{diag2}
\begin{tikzcd}
T(A)/n\ar{r}
\ar{d}{\pi_\star} & \Hom(\Br(A)_n/\Br_1(A)_n,\Q/\Z)\ar{d}{\widehat{\pi^\star}} & \\
 A_0(X)/n\ar{r}{\varepsilon_n^X} & \Hom(\Br(X)_n/\Br_0(X)_n,\Q/\Z),
\end{tikzcd}\end{equation}  
 where the top map is obtained by factoring $\varepsilon_n^A$ through $\Hom(\Br(A)_n/\Br_1(A)_n,\Q/\Z)$. 
Moreover, if $n$ is odd and the map $\varepsilon_n^A$ is injective, then the map $\pi_\star$ is an isomorphism, and the maps $\widehat{\pi^\star}$ and $\varepsilon_n^X$ are injective. 
\end{prop}
%
\begin{proof}
It follows by \cite[Prop. 2.16]{Gazaki2022weak}  that when the N\'{e}ron-Severi group $\NS(A_{\kk})$ has trivial $\Gal(\kk/k)$-action, the homomorphism $A_0(A)/n\xrightarrow{\varepsilon_n^A}\Hom(\Br(A)_n/\Br_0(A)_n,\Q/\Z)$ when restricted to $T(A)/n$ factors through the transcendental piece $\Hom(\Br(A)_n/\Br_1(A)_n,\Q/\Z)$. Moreover, Skorobogatov and Zharin (\cite[Theorem 2.4]{SZ2012}) showed that for every integer $n\geq 1$ the pullback $\pi^\star:\Br(X)\to\Br(A)$ induces an injection 
\[\pi^\star:\Br(X)_n/{\Br_1(X)_n}\hookrightarrow \Br(A)_n/\Br_1(A)_n,\] 
which is an isomorphism when $n$ is odd. By composing with the natural quotient map, we obtain a map
\[\pi^\star:\Br(X)_n/\Br_0(X)_n\to \Br(A)_n/\Br_1(A)_n\] 
which is a surjection when $n$ is odd. 
This shows that the diagram \eqref{diag2} is well-defined and its commutativity follows from the one of \eqref{diag1}. 


From now on assume that $n$ is odd and $\varepsilon_n^A$ is injective. As mentioned above, the map $\pi^\star$ is  surjective, and hence its dual $\widehat{\pi^\star}$ is injective.  
Thus, the composition $\widehat{\pi^\star}\varepsilon_n^A$ is injective, forcing the pushforward $\pi_\star$ to be injective as well. Notice that we have equalities \[T(A)/n=T(A)_{\nd}\{2'\}/n,\;A_0(X)/n=A_0(X)_{\nd}\{2'\}/n,\] and hence by \autoref{quotients} the map $\pi_\star$ is also surjective.  
 It follows that  $\pi_\star$ is an isomorphism and $\varepsilon_n^X$ is injective.

\end{proof}

\begin{rem}
	Even if $\varepsilon_n^A$ is not injective, we can note the following consequence of the proof: if for some odd $n$ there exists $z\in T(A)$ and $\alpha\in \Br_n(A)$ with $\langle z,\alpha\rangle\neq 0$, then $\pi_\star(z)$ is not contained in $n\cdot A_0(X)$.
\end{rem}

As we saw in the beginning of this subsection, injectivity of the homomorphism $\varepsilon_n^A$ is equivalent to injectivity of the cycle class map \[\CH_0(A)/n\xrightarrow{c_n^A}H^4_{\text{\'{e}t}}(A,\mu_n^{\otimes 2}).\]
For a product $A$ of elliptic curves the Albanese kernel $T(A)$ is a direct summand of the Chow group $\CH_0(A)$, and hence if $c_n$ is injective, then so is its restriction to the Albanese kernel. It is easy to see that the converse is also true (see for example \cite[Proposition 2.4]{Yamazaki2005}). Using the relation between $T(A)$ and the Somekawa $K$-group $K(k;E_1,E_2)$ attached to $E_1,E_2$ obtained in \cite[Theorem 2.2, Corollary 2.4.1]{Raskind/Spiess2000}, studying the map $c_n^A|_{T(A)/n}$ has been reduced to the study of the generalized Galois symbol (cf. \cite[Proposition 1.5]{Somekawa1990}) \[K(k;E_1,E_2)/n\xrightarrow{s_n}H^2(k,E_1[n]\otimes E_2[n]).\]
With these identifications, there are many situations where the injectivity of the cycle class map $c_n$ has been verified. Using these results together with \autoref{mainlocal} we show below that for all possible reduction types included in  \eqref{setup} we can find large enough finite extensions $K/k$ over which the group $A_0(X_K)_{\nd}$ is nontrivial. 

\begin{cor}\label{nontrivial} Consider the set-up \eqref{setup}  with $p>2$. Let $n\geq 1$. There is a finite extension $K/k$ such that the cycle class map 
\[c_{p^n}^X:A_0(X_K)/p^n\rightarrow H^4(X_K,\mu_{p^n}^{\otimes 2})\] is injective, and the group $A_0(X_K)/p^n$ is isomorphic to 
 $\Z/p^n$ or $(\Z/p^n)^2$ depending on the reduction type of $E_1, E_2$. 
\end{cor}
\begin{proof}
Consider a finite extension $K/k$ such that the N\'{e}ron-Severi group $\NS(A_{\kk})$ has trivial $\Gal(\kk/K)$-action, both elliptic curves have split semistable reduction and  $A[p^n]\subset A(K)$. The first assumption is possible because $\NS(A_{\kk})$ is a finitely generated abelian group. Depending on the reduction type of $E_1, E_2$ it follows by \cite[Theorem 1.2]{Yamazaki2005} and  \cite[Theorem 4.2]{Hiranouchi2014} that the cycle class map 
\[c_{p^n}^A:\CH_0(A_K)/p^n\rightarrow H^4(A_K,\mu_{p^n}^{\otimes 2})\] is injective.  Hiranouchi in \cite{Hiranouchi2014} in fact showed that the group $T(A_K)/p^n$ is isomorphic to $\Z/p^n$ if the two elliptic curves have the same reduction type (either good ordinary or split multiplicative reduction), and $(\Z/p^n)^2$ if the reduction types are different. It then follows by \autoref{mainlocal} that the same holds for the Kummer surface, $A_0(X_K)/p^n\simeq\Z/p^n$ or $A_0(X_K)/p^n\simeq(\Z/p^n)^2$. 

\end{proof}

The issue is that in general the extension $K/k$ considered above depends on the integer $n\geq 1$. 
In the case when both elliptic curves have good ordinary or split multiplicative reduction we can say something much stronger, obtaining a result independent of $n\geq 1$.
\begin{cor}\label{inectivity1} Let $A=E_1\times E_2$ be a product of two elliptic curves over $k$ and $X=\Kum(A)$ its associated Kummer surface. Suppose that both elliptic curves have either potentially split multiplicative reduction or potentially good ordinary reduction. Then there exists a finite extension $L/k$ such that for every odd $n\geq 1$ the cycle class map $c_{n}^{X_L}:\CH_0(X_L)/n\to H^4_{\text{\'{e}t}}(X_L,\mu_n^{\otimes 2})$ is injective and we have an isomorphism of finite groups \[\pi_{\star}:T(A_L)_{\nd}\{2'\}\xrightarrow{\simeq}A_0(X_L)_{\nd}\{2'\}. \] 
\end{cor}
\begin{proof}

We consider a finite extension $K/k$ such that the curves $E_{iK}$ have both either good ordinary or split multiplicative reduction, and the N\'{e}ron-Severi group $\NS(A_{\kk})$ has trivial $\Gal(\overline{K}/K)$-action. 

\textbf{Case 1:} Suppose that $E_{1K}, E_{2K}$ have split multiplicative reduction. It then follows by \cite[Theorem 1.2]{Yamazaki2005} that the cycle class map $c_n^{A_K}$ is injective for every $n\geq 1$. 

\textbf{Case 2:} Suppose that $E_{1K}, E_{2K}$ have good ordinary reduction. For an integer $m$ coprime to $p$ it follows by \cite[Theorem 3.5]{Raskind/Spiess2000} that the group $T(A)$ is $m$-divisible, which implies that the cycle class map $c_m^{A_K}$ is injective. We next consider the $p$-part. It follows by \cite[Theorem 1.4]{Gazaki/Leal2022} that there exists a finite extension $L/K$ such that the cycle class map $c_{p^n}^{A_L}:\CH_0(A_L)/p^n\to H^4_{\text{\'{e}t}}(A_L,\mu_{p^n}^{\otimes 2})$ is injective for all $n\geq 1$. 

 Fix an extension $L/k$ as above and $n\geq 1$ an odd integer. 
  \autoref{mainlocal} gives that the map $c_n^{X_L}$ is injective for all $n\geq 1$ and the push-forward 
 \[\pi_\star: T(A_L)/n\to A_0(X_L)/n\] is an isomorphism. 
  To conclude that $\pi_\star$ induces an isomorphism $T(A_L)_{\nd}\{2'\}\xrightarrow{\simeq}A_0(X_L)_{\nd}\{2'\}$ we apply the above for $n=N=|T(A_L)_{\nd}\{2'\}|$  (recall this value is finite by \cite[Theorem 1.2]{Gazaki/Leal2022}). 
 
%


\end{proof}

Using results obtained in \cite{Gazaki/Leal2022} we can describe cases when the conclusions of \autoref{inectivity1} hold already over the base field $k$, allowing us to fully compute the group $A_0(X)_{\nd}\{2'\}$.  

\begin{cor}\label{injectivity2} Let $p>2$,  let $A=E_1\times E_2$ be a product of two elliptic curves over $k$ with \textbf{good ordinary reduction,} and let $X=\Kum(A)$. Assume that $A[p]\subset A(k)$ and let $n\geq 1$ be the largest positive integer such that $A[p^n]\subset A(k)$. Suppose that the N\'{e}ron-Severi group $\NS(A_{\overline{k}})$ has trivial $\Gal(\overline{k}/k)$-action and that the extension $k(A[p^{n+1}])/k$ has wild ramification. Then the cycle class maps $c_{p^m}^A, c_{p^m}^X$ are injective for all $m\geq 1$ and we have isomorphisms 
\[A_0(X)_{\nd}\{2'\}\simeq T(A)_{\nd}\simeq \Z/p^n.\]   
\end{cor}
\begin{proof}
Everything follows from \cite[Corollary 3.17]{Gazaki/Leal2022}) and \autoref{inectivity1}.  

\end{proof}



The wild ramification assumption is very often satisfied. Namely, if we start with a product $A=E_1\times E_2$ of elliptic curves with good ordinary reduction over $\Q_p$ and set $k_n=\Q_p(A[p^n])$, then the extension $k_{n+1}/k_n$ has almost always wild ramification. The following Corollary gives a specific case when this is true and additionally the assumption on the N\'{e}ron-Severi group is satisfied. We will revisit this case when we study local-to-global principles in the next section. 

\begin{cor}\label{CMcase} 
	 Let $p>2$ and let $A=E\times E$ be the self product of an elliptic curve $E$ over $\Q_p$ with good ordinary reduction.  Suppose that $E$ has complex multiplication by the full ring of integers of a quadratic imaginary field $K$, with all endomorphisms defined over $\Q_p$.  
	For each $n\geq 1$ let $k_n=\Q_p(E[p^n])$. Then we have isomorphisms 
	\[A_0(X_{k_n})_{\nd}\{2'\}\simeq  T(A_{k_n})_{\nd} \simeq\Z/p^n.\]
\end{cor}
\begin{proof}
It follows by \cite[Proposition 2.17]{Gazaki2022weak} that the assumption on the N\'{e}ron-Severi group is satisfied in this case. Moreover, for each $n\geq 1$ the extension $k_n=\Q_p(E[p^n])$ has absolute ramification index $e_{k_n}=(p-1)p^{n-1}$. In particular the extension $k_{n+1}/k_n$ has wild ramification. For more details on the above claim see \cite[section 3.2.1]{Gazaki2022weak}. 

\end{proof}

\subsection{Behavior over the maximal unramified extension}\label{maxunr} Let $A=E_1\times E_2$ be a product of elliptic curves over a $p$-adic field $k$ with good ordinary or almost ordinary reduction and let $X=\Kum(A)$. It follows that $A_0(X)_{\nd}$ is a finite $p$-group. We note that even when this group is nontrivial, it seems that it becomes trivial after extending to the maximal unramified extension $k^{\ur}$ of $k$. This is at least true in the case we are in the situation of \autoref{injectivity2}. 
To see this, suppose that $A_0(X)_{\nd}\simeq T(A)_{\nd}\simeq\Z/p^n$ for some $n\geq 1$. Let $k'/k$ be a finite unramified extension. It follows by \autoref{injectivity2} that $A_0(X_{k'})_{\nd}\simeq T(A_{k'})_{\nd}\simeq\Z/p^n$. Moreover, it follows by Lemma 4.3 in \cite{Gazaki/Hiranouchi2021} and its proof (see also Lemma 4.1) that the norm map $N_{k'/k}:T(A_{k'})\to T(A)$ is surjective. We conclude that $N_{k'/k}$ induces an isomorphism 
\[N_{k'/k}:A_0(X_{k'})_{\nd}\xrightarrow{\simeq} A_0(X)_{\nd}.\]
Since $N_{k'/k}\res_{k'/k}=[k':k]$, if we take $k'/k$ to be the unique unramified extension of degree $p^n$ we see that the restriction map 
\[\res_{k'/k}:A_0(X_{k})_{\nd}\to A_0(X_{k'})_{\nd}\] is zero. We expect this to be true more generally in the case of good reduction. The divisibility of $A_0(X_{k^{\ur}})$ was suggested as a logical possibility in \cite{Esnault/Wittenberg} in the paragraph following Theorem 5.1, where the authors proved that for every semistable $K3$ surface over $\C((t))$ the group $A_0(X)$ is divisible. In the current article we cannot say anything definitive  
about the behavior over $k^{\ur}$ in the case of bad semistable reduction.

\vspace{3pt}

\section{Weak Approximation for zero-cycles}\label{weaksection} In this section we discuss applications of the local results obtained in sections \eqref{section2} and \eqref{localresults} to local-to-global principles for zero-cycles. Throughout this section we will be working over an algebraic number field $F$.  We will denote by $\Omega$ the set of all places of $F$ and by $\Omega_f$ the set of all finite places. For a place $v\in\Omega$ we will denote by $F_v$ the completion of $F$ at $v$. Moreover, for a variety $X$ over $F$ we will denote by $X_v:=X\otimes_F F_v$ the base change to $F_v$. For a finite extension $K/F$ we will denote by $\Omega_K$ the set of places of $K$. 


For a smooth projective variety $Y$ over $F$ the \textit{adelic Chow group} $\CH_{0,\mathbf{A}}(Y)$ is defined to be the infinite product $\displaystyle\prod_{v\in\Omega}\CH_0'(Y_v)$, where $\CH_0'(Y_v)=\CH_0(Y_v)$ for every finite place $v$, while at infinite places $\CH_0'(Y_v)$ is the cokernel of the norm (push-forward) map 
\[N_{\overline{F}_v/F_v}:\CH_0(Y_v\otimes_{F_v}\overline{F}_v)\to \CH_0(Y_v).\] It follows that $\CH_0'(Y_v)=0$ for all infinite complex places, while for the infinite real places the group $\CH_0'(Y_v)$ is $2$-torsion (see \cite[Th\'{e}or\`{e}me 1.3]{Colliot-Thelene1995}). We define $A_{0,\mathbf{A}}(Y)$ and $A_0'(Y_v)$ in a similar way. \black 

For each finite place $v\in \Omega_f$ we consider the Brauer-Manin pairing (cf.~\autoref{computations1}) 
\[\langle\cdot,\cdot\rangle_v:\CH_0(Y_v)\times\Br(Y_v)\to\Br(F_v)\simeq\Q/\Z. \]
 Such a pairing can be defined also for every real place $v$ (by embedding $\Br(\mathbb{R})\simeq\Z/2$ to $\Q/\Z$), and in this case the group $N_{\overline{F}_v/F_v}(\CH_0(Y_v\otimes_{F_v}\overline{F}_v))$ lies in the left kernel of $\langle\cdot,\cdot\rangle_v$. 

The local pairings induce a global pairing,
\[\langle\cdot,\cdot\rangle:\CH_{0,\mathbf{A}}(X)\times\Br(X)\rightarrow\Q/\Z,\] defined by $\langle(z_v)_v,\alpha\rangle=\sum_v\langle z_v,\iota_v^\star(\alpha)\rangle_v$, where $\iota_v^\star$ is the pullback of $\iota_v: X_v\to X$. 
The short exact sequence of global class field theory, \[0\rightarrow\Br(F)\rightarrow\bigoplus_{v\in\Omega}\Br(F_v)\xrightarrow{\sum\inv_v}\Q/\Z\rightarrow 0,\] implies that the group $\CH_0(X)$ lies in the left kernel of $\langle\cdot,\cdot\rangle$.

We are interested in the following conjecture due to Colliot-Th\'{e}l\`{e}ne and Sansuc for geometrically rational varieties and to Kato and S. Saito for general smooth projective varieties. 
\begin{conj}\label{localtoglobal2} (\cite[Conjecture A] 
	{Colliot-Thelene/Sansuc1981}, \cite[Section 7]{Kato/Saito1986}, see also \cite[Section 1.1]{Wittenberg2012})  Let $Y$ be a smooth, projective, geometrically connected variety over a number field $F$. The following complex is exact:
\begin{equation}\label{complex1}\widehat{\CH_0(Y)}\stackrel{\Delta}{\longrightarrow}
\widehat{\CH_{0,\mathbf{A}}(Y)}\rightarrow\Hom(\Br(Y),\Q/\Z).
\end{equation}
\end{conj} 
For a general $K3$ surface $X/F$ Ieronymou (\cite{Ieronymou2021}) obtained the following result, which is a first step towards \autoref{localtoglobal2}. 
\begin{theo}\label{Ieronymou} (\cite[Theorem 1.3]{Ieronymou2021}) Let $X$ be a $K3$ surface  over $F$. Suppose that for all finite extensions $L/F$ the set $X(L)$ is dense in $X(\mathbf{A}_L)^{\Br(X_L)}$. Let $\{z_v\}_{v\in\Omega}$ be a family of local zero-cycles of constant degree orthogonal to $\Br(X)$. Then for every $n\geq 1$ there exists a zero-cycle $z\in\CH_0(X)$ of the same degree such that $z\equiv z_v$ in $\CH_0(X_v)/n$ for all $v\in\Omega$. 
\end{theo}

In what follows we focus on cycles of degree $0$. For a variety $Y$ over $F$ \autoref{localtoglobal2} would imply that the following complex is exact 
\begin{equation}\label{complex2}\widehat{A_0(Y)}\stackrel{\Delta}{\longrightarrow}
\widehat{A_{0,\mathbf{A}}(Y)}\rightarrow\Hom(\Br(Y)/\Br_0(Y),\Q/\Z).
\end{equation} 
It is easy to see that if $Y(F)\neq\emptyset$, then the converse is also true. Even when this is not the case, one may focus only on the exactness of \eqref{complex2}. 

In this section we want to strengthen \autoref{Ieronymou} for geometrically split Kummer $K3$ surfaces in two different directions. Namely, we focus on the following questions: 
\begin{enumerate}
\item[1.] Fixing a prime $p$, we may apply \autoref{Ieronymou} for $p^n$, where $n\geq 1$. When can we pass to the limit for $n$? 
\item[2.] Can we obtain any unconditional results? 
\end{enumerate}


To elaborate more on the first question, if we fix a prime number $p$, we can ask whether the following complex is exact 
\begin{equation}\label{complexp}\varprojlim\limits_{n}A_0(Y)/p^n\stackrel{\Delta}{\longrightarrow}
\prod_{v\in\Omega}\varprojlim\limits_{n}A_0'(Y_v)/p^n\rightarrow\Hom(\Br(Y)\{p\}/\Br_0(Y)\{p\},\Q/\Z). 
\end{equation}
Note that proving exactness of complex \eqref{complex2} is equivalent to proving exactness of \eqref{complexp} for every prime $p$.  Our first result is the following.

\begin{theo}\label{localglobal1} 
Let $X$ be a geometrically split Kummer $K3$ surface over a number field $F$. 
There exists an infinite set $T$ of prime numbers $p$ for which the following are true:
\begin{enumerate}
\item[(a)]  For every $p\in T$ the middle term of the complex \eqref{complexp} is the following finite group \[\prod_{v|p}A_0(X_v)_{\nd}\{p\}.\] In particular, it consists of genuine local zero-cycles. 
\item[(b)] Suppose that for every finite extension $L/F$ the set $X(L)$ is dense in the Brauer set $X(\mathbf{A}_L)^{\Br(X_L)}$. Then for every $p\in T$ the complex \eqref{complexp} is exact.
\item[(c)] For all but finitely many $p\in T$ the middle term of the complex \eqref{complexp} vanishes and for these primes the density assumption in (b) can be dropped. 
\end{enumerate} 
\end{theo}
\begin{proof}
The proof will be along the lines of \cite[Theorem 3.1]{Gazaki2022weak}. We will construct the required infinite set of primes $p$ by imposing certain conditions. Let $K/F$ be a finite extension such that there exists an abelian surface $A$ over $K$ admitting an isogeny $\phi: E_1\times E_2\to A$ and such that there exists an isomorphism $X_K\simeq A_K$. Without loss of generality we can take $K/F$ to be Galois, so that for any $w\in \Omega_K$ lying over a place $v\in \Omega$ we have $[K_w:F_v]\mid [K:F]$. 
Let $\Lambda_K\subset \Omega_K$ be the set of all finite places of bad reduction of the abelian surface $A$. Note that $\Lambda_K$ is a finite set. Consider the following finite set of places of $F$:  
\[\Lambda=\{v\in\Omega_f: \exists w\in\Lambda_K \text{ such that }w|v\}.\]

\textbf{Condition 1:} Let $p$ be prime not dividing $2\deg(\phi)[K:F]$ and such that for every place $w\in\Omega_K$ above $p$ the abelian surface $A_w$ has good reduction.

We will show that for this collection of primes the middle term of \eqref{complexp} becomes 
\begin{equation}\label{eqn1}
\left(\prod_{v|p}\varprojlim\limits_n A_0(X_v)/p^n\right)\times \left(\prod_{v\in\Lambda}\varprojlim\limits_n A_0(X_v)/p^n\right).
\end{equation}

 Let $v\in\Omega$ be such that $v\not\in\Lambda$ and $v\nmid p$. The claim will follow if we show that the group $A_0'(X_v)$ is $p$-divisible. Suppose first that $v$ is an infinite real place. Then the group $A_0'(X_v)$ is $2$-torsion, and hence $p$-divisible since we assumed that $p$ is odd. Next suppose that $v\in\Omega_f$. 
The assumption  $v\not\in\Lambda$ implies that for every finite place $w$ of $K$ lying over $v$ the abelian surface $A_w$ has good reduction. Thus, so does the $K3$ surface $X_w\simeq\Kum(A_w)$ by \cite[Lemma 4.2]{Matsumoto2015}. The assumption $v\nmid p$ implies that $v$ lies above a rational prime $l\neq p$. It then follows by  \cite[Proposition 3.4]{Ieronymou2021} 
that the group $A_0(X_w)$ is $p$-divisible. 
  Since $p$ is coprime to $2\deg(\phi)[K_w:F_v]$, the norm map 
  \[N_{K_w/F_v}:A_0(X_w)/p\to A_0(X_v)/p\]
  is a surjective map of $\F_p$-vector spaces, since $N_{K_w/F_v} \res_{K_w/F_v}=[K_w:F_v]$. We conclude that $A_0'(X_v)=A_0(X_v)$ is $p$-divisible as desired. 



\textbf{Condition 2:} We require additionally that for every finite place $w\in\Omega_K$ lying above $p$ the abelian surface $A_w$ has good ordinary or almost ordinary reduction.  

Since we assumed that $p$ is coprime to $\deg(\phi)$, this condition is equivalent to requiring that at least one of the elliptic curves $E_{1w}, E_{2w}$ has good ordinary reduction. 
Under this condition, it follows directly by \autoref{padic1} that for each finite place $v\in\Omega$ lying above $p$ the group $A_0(X_v)_{\nd}$ is finite. This yields an equality 
\[\prod_{v|p}\varprojlim\limits_n A_0(X_v)/p^n=\prod_{v|p}A_0(X_v)_{\nd}\{p\}.\]
In particular, the first component of \eqref{eqn1} consists of genuine zero-cycles. 
We need one last condition that will guarantee that the second component of \eqref{eqn1} vanishes. 

Set $B=E_1\times E_2$. For each $w\in\Lambda_K$ there exists a finite extension $L^w/K_w$ such that the abelian surface $B_{L^w}$ attains split semistable reduction; in particular we can find such an extension with index $[L^w:K_w]$ dividing $24$ \cite[VII.5.4(c), A.1.4a]{Silverman}. \black
We will say that $B_w$ has \textit{potentially good reduction} if $B_{L^w}$   has good reduction. This is equivalent to both elliptic curves $E_{i,L^w}$ having good reduction. Otherwise we will say that $B_w$  has potentially split bad reduction, which is equivalent to saying that at least one of the elliptic curves  $E_{i,L^w}$ has split multiplicative reduction. Let $\Lambda_{1}\subset\Lambda$ be the set of  places of potentially good reduction and $\Lambda_{2}\subset\Lambda$ the set of places of potentially split bad reduction. 
For $v\in\Lambda_2$, it follows by \autoref{padic1} that the group $A_0(X_v)_{\nd}$ is finite; let $N_v$ denote its order. Consider the following positive integer, 
\[M:=6\prod_{v\in\Lambda_2}N_v.\] 

\textbf{Condition 3:} In addition to conditions 1 \& 2, we also require that $p$ is coprime to $M$. 

It is enough to show that if $p\nmid M$, then for every $v\in\Lambda$ the group $A_0(X_v)$ is $p$-divisible. For $v\in\Lambda_2$ this follows by the assumption $p\nmid N_v$.  For $v\in\Lambda_1$ this follows similarly to the good reduction case above. That is, since $v\nmid p$ and $p\nmid 2\deg(\phi)$, the group $A_0(X_{L^w})$  is $p$-divisible. 
 Since $p\nmid 6$ we have $p\nmid [L^w:K_w]$, which together with $p\nmid [K_w:F_v]$ implies $A_0(X_v)$ is $p$-divisible by the usual norm-restriction argument. \black 

Let $T$ be the set of prime numbers $p$ that satisfy conditions 1-3. Then for every $p\in T$ the middle term $\displaystyle \prod_{v\in\Omega}\varprojlim\limits_n A_0'(X_v)/p^n$ of the complex \eqref{complexp} becomes
\[\prod_{v|p}A_0(X_v)_{\nd}\{p\},\] which is a finite group consisting of genuine zero-cycles. This proves (a). Let $p^N$ be the order of this finite group. 

We next prove (b). Let $\displaystyle\{z_v\}_{v\in\Omega}\in \displaystyle \prod_{v\in\Omega}\varprojlim\limits_n A_0'(X_v)/p^n$ be a family of local zero-cycles which is orthogonal to $\Br(X)$, where we may assume that $z_v=0$ for every $v\nmid p$. 
It follows by \autoref{Ieronymou} that there exists a global zero-cycle $z\in A_0(X)$ such that $z\equiv z_v$ in $A_0(X_v)/p^N$ for all $v\in\Omega$. But  for every place $v\nmid p$, $A_0(X_v)/p^N=A_0(X_v)_{\nd}\{p\}$.
 We conclude that under the diagonal map $\Delta$ the image of $z$ coincides with the local family $(z_v)$, and hence the complex \eqref{complexp} is exact. 

Lastly we prove (c).
 Let $T_0\subseteq T$ be the subset that contains all primes $p$ such that for each place $w\in\Omega_K$ with $w\mid p$ 
 the extension $K_w/\Q_p$ is unramified. Since there are only finitely many rational primes that ramify in $K$, $T_0$ contains all but finitely many of the primes in $T$. For $p\in T_0$,  it follows by \autoref{pdivisible} that each $A_0(X_v)_{\nd}\{p\}$ is trivial. Thus, in this case exactness of the complex \eqref{complexp} follows by vanishing of the middle term. 

We close this proof by noting that the set $T$ we constructed is always infinite, since it contains all but finitely many of the  primes $p$ lying below places of $K$ of good ordinary reduction, and the latter is always infinite. In fact, if we additionally assume that at least one of the elliptic curves has geometric endomorphism ring $\Z$, then the complement of $T$ is a set of primes of density zero (see \cite{LangTrotter, Serre1981}). 

\end{proof} 

\begin{rem} We could have chosen to state the above theorem in a simpler way.  Namely, by discarding the finitely many primes from the set $T$ for which the middle term $\displaystyle\prod_{v|p}A_0(X_v)_{\nd}\{p\}$ might be nontrivial, we would have a statement in (c) true for every prime in $T$. The reason we chose to state \autoref{localglobal1} for this enlarged set $T$ is because in the next section we will use the explicit description of the middle term in cases when it becomes nontrivial. In those cases statement (b) of \autoref{localglobal1} becomes relevant. 
\end{rem}

\begin{rem}\label{BB}
The key point in the above proof that allowed us to pass to the limit for $p^n$ in the result of Ieronymou was the $p$-divisibility results obtained in \autoref{unramified1} and \autoref{pdivisible}. 
In \cite[Conjecture 1.3]{Gazaki/Hiranouchi2021} we proposed a conjecture that for a product $A=E_1\times E_2$ of elliptic curves over a number field $F$, $T(A_v)=0$ for all unramified places of good reduction. The cases we are missing are the places where both elliptic curves have good supersingular reduction. If we could prove this, 
then this would imply  \autoref{localtoglobal2} for Kummer surfaces associated to products of elliptic curves in its greatest generality, since any such $K3$ surface has a rational point. One could also try to remove the strong assumption on the Brauer-Manin obstruction being the only obstruction to weak approximation for rational points. For, the proof of \autoref{localglobal1} would reduce the problem to proving exactness of the complex \eqref{complexp} for only the following finite list of primes: 
\begin{itemize}
\item $p|2M$, where $M$ is the integer defined in the proof of \autoref{localglobal1}, 
\item primes $p$ lying under the set $\Lambda$ of bad reduction places,
\item primes that ramify in $F$. 
\end{itemize}
Thus, one could try to find explicit examples of Kummer surfaces over $\Q$ that satisfy the conjecture unconditionally. 
We close this remark by noting that the expectation $A_0(X_v)=0$ for a Kummer surface $X$ and for all (but finitely many) unramified places of good reduction is also motivated by the Bloch-Beilinson conjectures. For a general smooth projective variety $Y$ over a number field Bloch (\cite{Bloch1983}) predicts that the Chow group $\CH_0(Y)$ is finitely generated, while Beilinson (\cite{Beilinson1984}) predicts that $T(Y_{\overline{\Q}})=0$. Putting these together, the expectation is that the Albanese kernel $T(Y)$ is a finite group, and hence it should equal its completion. If now $Y$ is a $K3$ surface, then it follows by \cite[Theorem 1.2]{Skorobogatov/Zharin2008} that the group $\Br(Y)/\Br_0(Y)$ is finite. 
Thus, the complex 
\[\widehat{A_0(Y)}\stackrel{\Delta}{\longrightarrow}
\widehat{A_{0,\mathbf{A}}(Y)}\rightarrow\Hom(\Br(Y)/\Br_0(Y),\Q/\Z)
\] being exact would force the group $\widehat{A_{0,\mathbf{A}}(Y)}$ to be finite, and hence only finitely many places should contribute nontrivial terms.

\end{rem}


\vspace{2pt}

\subsection{The role of good reduction places}\label{goodredplaces} In this section, building from the main results in \cite{Gazaki2022weak},  we will construct explicit examples of Kummer surfaces to show:  
\begin{itemize}
\item The places of good ordinary reduction can contribute nontrivially to the Brauer set for zero-cycles of degree $0$, as long as we pass to a suitable finite ramified extension. 
\item Obtain unconditional evidence for the exactness of the complex \eqref{complexp}, even in cases when there are nontrivial local families of zero-cycles that are orthogonal to the Brauer group. 
\end{itemize}


We start with the following lemma, which follows by an easy diagram chasing. 
\begin{lem}\label{diagramchasing} Let $A_i\xrightarrow{f_i}B_i\xrightarrow{g_i}C_i$ be complexes of abelian groups for $i=1,2$ and suppose there is a commutative diagram
\[\begin{tikzcd}
A_1\ar{r}{f_1}\ar{d}{\alpha} & B_1\ar{r}{g_1}\ar{d}{\beta} & C_1\ar{d}{\gamma} \\
A_2\ar{r}{f_2} & B_2\ar{r}{g_2} & C_2.
\end{tikzcd}\] 
\begin{enumerate}
\item[(i)] If $\beta$ is surjective, $\gamma$ is injective, and the top complex is exact, then the bottom complex is exact. 
\item[(ii)] If $\alpha$ is surjective, $\beta$ is injective, and the bottom complex is exact, then the top complex is exact. 
\end{enumerate}
\end{lem}
\begin{proof}
First suppose $\beta$ is surjective,  $\gamma$ is injective, and $A_1\xrightarrow{f_1}B_1\xrightarrow{g_1}C_1$ is exact. Let $b_2\in\ker(g_2)$. Since $\beta$ is surjective, there exists $b_1\in B_1$ such that $\beta(b_1)=b_2$. Then $\gamma(g_1(b_1))=0$, and $\gamma$ is injective, which yields $b_1\in\ker(g_1)=\img(f_1)$. Thus we may write $b_1=f_1(a_1)$ for some $a_1\in A_1$. Commutativity yields $f_2(\alpha(a_1))=\beta(f_1(a_1))=\beta(b_1)=b_2$. Hence $b_2\in\img(f_2)$ as required. 

Next assume that $\alpha$ is surjective, $\beta$ is injective, and that $A_2\xrightarrow{f_2}B_2\xrightarrow{g_2}C_2$ is exact. Let $b_1\in\ker(g_1)$. Then $g_2(\beta(b_1))=0$, and hence $\beta(b_1)\in\ker(g_2)=\img(f_2)$. Write $\beta(b_1)=f_2(a_2)$ for some $a_2\in A_2$. Since $\alpha$ is surjective, we can find $a_1\in A_1$ such that $a_2=\alpha(a_1)$. Then $\beta(f_1(a_1))=f_2(\alpha(a_1))=f_2(a_2)=\beta(b_1)$. Since $\beta$ is injective we conclude that $b_1=f_1(a_1)$, and hence $b_1\in\img(f_1)$ as required. 

\end{proof} 
 
 \begin{ass} From now on we assume that $A=E_1\times E_2$ is a product of two elliptic curves over a number field $F$, and $X=\Kum(A)$ is the associated Kummer surface. Additionally, we assume that  the N\'{e}ron-Severi group $\NS(A_{\overline{F}})$ has trivial $\Gal(\overline{F}/F)$-action.
 \end{ass}

 Let $p$ be an odd prime. It follows by \cite[Proposition 2.16]{Gazaki2022weak} that the complex \eqref{complexp} induces a complex 
\begin{equation}\label{complexp*}\varprojlim\limits_{n}T(A)/p^n\stackrel{\Delta}{\longrightarrow}
\prod_{v\in\Omega_f}\varprojlim\limits_{n}T(A_v)/p^n\rightarrow\Hom(\Br(A)\{p\}/\Br_1(A)\{p\},\Q/\Z).
\end{equation}

Moreover, tracking what the proof of \autoref{localglobal1} gives in the special case of a Kummer surface (that is $K=F$ and $\deg(\phi)=1$), we see that for every prime number $p$ in the set $T$ constructed in the proof, it follows that the above complex becomes 
\begin{equation}\label{complexp**}\varprojlim\limits_{n}T(A)/p^n\stackrel{\Delta}{\longrightarrow}
\prod_{v|p}T(A_v)_{\nd}\{p\}\rightarrow\Hom(\Br(A)\{p\}/\Br_1(A)\{p\},\Q/\Z).
\end{equation}

\begin{lem}\label{commutes} Consider the above set-up. For every prime number $p\in T$ we have a commutative diagram 
\[\begin{tikzcd}
T(A)/p^N\ar{r}{\Delta_A}\ar{d}{\pi_\star} & \prod_{v|p}T(A_v)_{\nd}\{p\}\ar{r}{\varepsilon_A}\ar{d}{(\pi_{v\star})_v} &\Hom(\Br(A)\{p\}/\Br_1(A)\{p\},\Q/\Z)\ar{d}{\widehat{\pi^\star}}\\
A_0(X)/p^N\ar{r}{\Delta_X} & \prod_{v|p}A_0(X_v)_{\nd}\{p\}\ar{r}{\varepsilon_X} & \Hom(\Br(X)\{p\}/\Br_0(X)\{p\},\Q/\Z),
\end{tikzcd}\] where $p^N$ is the order of the finite group $\prod_{v|p}T(A_v)_{\nd}\{p\}$. In this diagram the two leftmost vertical maps are surjective and the rightmost vertical map is  injective.  
Thus, if the top complex is exact, then so is the bottom one. 
\end{lem}

\begin{proof}
The commutativity of the left square is clear. Moreover, surjectivity of the first two vertical maps follows by \autoref{coprimeto2}. The maps $\varepsilon_A,\varepsilon_X$ are defined by summing the local pairings, for which commutativity follows by \autoref{mainlocal}. Thus, the right square is commutative. Lastly, since every $p\in T$ is odd, the injectivity of  
$\widehat{\pi^\star}$ also follows from \autoref{mainlocal}.  
The rest follows by \autoref{diagramchasing}.

\end{proof}


\subsection*{Kummer surface associated to a self-product of a CM Elliptic Curve} 

\begin{notn} For an elliptic curve $E$ over $\Q$ and a prime $p$ of good reduction we will denote by $\overline{E_p}$ the reduction of $E$ modulo $p$, which is an elliptic curve over $\F_p$. 
\end{notn}

The following proposition is the analog of \cite[Theorem 4.2]{Gazaki2022weak}), which for simplicity we only consider in a special case. We note that there is a mistake in the statement of the aforementioned theorem, which has now been fixed (see \autoref{fix} for details). The following is the analog of the corrected version of \cite[Theorem 4.2]{Gazaki2022weak}. 

\begin{prop}\label{Brorthogonal} Let $E$ be an elliptic curve over $\Q$ such that $E_{\overline{\Q}}$ has complex multiplication by the full ring of integers $\mathcal{O}_K$ of a quadratic imaginary field $K$. Let $A=E\times E$ and $X=\Kum(A)$. Let $p\geq 5$ be a prime that splits completely in $K$ and is coprime to the conductor $\mathfrak{n}$ of $E$. Suppose that the reduction $\overline{E_p}$ of $E$ modulo $p$ has the property that $|\overline{E_p}(\F_p)|=p$. Then there exists a finite extension $L/K$ of degree $p-1$ such that for the Kummer surface $X_L$ the following are true:
\begin{enumerate}
\item[(i)] The Brauer group $\Br(X_L)\{p\}/\Br_0(X_L)\{p\}$ vanishes. 
\item[(ii)] The middle term of the complex \eqref{complexp} becomes $\displaystyle\prod_{v|p}A_0(X_{L_v})_{\nd}\{p\}=\Z/p$. 
\end{enumerate}
\end{prop}
\begin{proof}
 We recall the following facts about elliptic curves over $\Q$ with CM (for more details we refer to \cite[Section 2]{Gazaki2022weak}). 
 The field $K$ is one of the nine quadratic imaginary fields with class number $1$. That is, $\mathcal{O}_K$ is a PID. Let $p$ be a prime that splits completely in $K$ with $p$ coprime to $\mathfrak{n}$. We may write $p=\eta\overline{\eta}$ for some prime element $\eta$ of $K$. The elliptic curve $E$ has good ordinary reduction at $p$. In fact, the element $\eta$ can be chosen so that the endomorphism $[\eta]: E_{\Q_p}\to E_{\Q_p}$ when reduced modulo $p$ coincides with the Frobenius endomorphism $\overline{E_p}\xrightarrow{\phi}\overline{E_p}$. We denote by $E[\eta]$ the kernel of $[\eta]:E_K\to E_K$, and by $\mathfrak{p}=(\eta), \overline{\mathfrak{p}}=(\overline{\eta})$ the primes of $K$ above $p$. We note that in \cite{Gazaki2022weak} we used $\pi$ to denote the prime element of $K$ that reduces to the Frobenius. In this article we changed notation to $\eta$ to avoid confusion with the projection map $\pi:A\dashrightarrow X$. 
 
From now on assume that $p\geq 5$ and that $|\overline{E_p}(\F_p)|=p$. We consider the extension $L=K(E[\eta])$. It follows by \cite[Corollary 5.20 (ii), (iii)]{Rub99} that $L/K$ is a cyclic Galois extension of order $p-1$ and it is totally ramified above $\mathfrak{p}$, and unramified above $\overline{\mathfrak{p}}$. Let $v$ be the unique place of $L$ lying above $\mathfrak{p}$. It follows by \cite[Theorem 4.2]{Gazaki2022weak}) and its proof that the following are true for the abelian surface $A_L$:  \begin{enumerate}
\item[(a)] $\Br(A_L)\{p\}/\Br_1(A_L)\{p\}=0$.  
\item[(b)] $\displaystyle\prod_{w|p}T(A_{L_w})_{\nd}\{p\}=T(A_{L_v})_{\nd}\{p\}\simeq\Z/p$.
\end{enumerate}
Since $p$ is odd, it follows from the proof of \autoref{mainlocal}  that $\pi^\star: \Br(X_L)\{p\}/\Br_0(X_L)\{p\}\to \Br(A_L)\{p\}/\Br_1(A_L)\{p\}$ is a surjection, \black 
 from where the vanishing $\Br(X_L)\{p\}/\Br_0(X_L)\{p\}=0$ follows.

To show that the middle term of the complex \eqref{complexp} becomes $\displaystyle\prod_{v|p}A_0(X_{L_v})_{\nd}\{p\}$, it is enough to verify that every prime $p\geq 5$ of good reduction that splits completely in $K$ lies in the set $T_L$ of \autoref{localglobal1} corresponding to the Kummer surface $X_L$.
 We have already established the reduction type. It remains to show $p\nmid M$. This follows by the assumption $p\geq 5$ (see Claim 2 in the proof of \cite[Theorem 4.2]{Gazaki2022weak}). 

To finish the proof of (ii), it suffices to show that the map \[\pi_{v\star}: T(A_{L_v})_{\nd}\{p\}\to A_0(X_{L_v})_{\nd}\{p\}\] is an isomorphism. This will follow from \autoref{CMcase} as long as we establish the following. 

\textbf{Claim:} The extension $L_v/\Q_p$ coincides with the extension $\Q_p(E_{\Q_p}[p])$.


This follows by the proof of \cite[Theorem 4.2]{Gazaki2022weak}). Namely, the CM assumption implies that there is a \textbf{splitting short-exact sequence} of $\Gal(\overline{\Q}_p/\Q_p)$-modules 
\[0\to\widehat{E_{\Q_p}}[p]\to E_{\Q_p}[p]\to\overline{E_p}[p]\to 0,\] where $\widehat{E_{\Q_p}}$ is the formal group associated to the elliptic curve $E_{\Q_p}$. The Galois module $\widehat{E_{\Q_p}}[p]$ coincides with $E[\eta]$, and hence $L_v\subset \Q_p(E_{\Q_p}[p])$. The assumption $|\overline{E_p}(\F_p)|=p$ is equivalent to saying that $\overline{E_p}[p]= \overline{E_p}(\F_p)$, from where the claim follows.

\end{proof} 

\begin{rem}\label{fix} The current statement of \cite[Theorem 4.2]{Gazaki2022weak}) states that for the abelian surface $A$ the middle term $\displaystyle\prod_{w|p}T(A_{L_w})_{\nd}\{p\}\simeq(\Z/p)^2$. The  mistake stemmed from the wrong assumption that the extension $L/\Q$ is Galois, which would force the extension $L/K$ to be totally ramified above the place $\overline{\mathfrak{p}}$. This is wrong however. It follows by \cite[Corollary 5.20]{Rub99} that $L/K$ is unramified above $\overline{\mathfrak{p}}$, which yields a vanishing $\displaystyle\prod_{w|p, w\neq \mathfrak{p}}T(A_{L_w})_{\nd}\{p\}=0$. This mistake has been addressed by E. Gazaki's PhD student, M. Wills, in \cite[Remark 3.16]{Wills25}. 
\end{rem}

\begin{exmp}\label{example} In \cite[Examples 2.4, 2.6, 2.7]{Gazaki2022weak}) we constructed many explicit examples of elliptic curves $E$ and primes $p$ satisfying the assumptions of \autoref{Brorthogonal}. We recall some of these examples here. 
\begin{itemize}
\item Let $K=\Q(\sqrt{-3})$. Consider the family of elliptic curves $\{E_c:y^2=x^3+c, c\in\Z, c\neq 0\}$ having CM by $\Z[\zeta_3]$. Let $p$ be a prime of the form $4p=1+3v^2$ (e.g. $p=7, 37, 61$). Then for exactly $\displaystyle\frac{p-1}{6}$ congruence classes of $c\mod p$ the elliptic curve $E_c$ satisfies $|\overline{E_{c,p}}(\F_p)|=p$.  
\item Let $K=\Q(\sqrt{-11})$ and $p=223$. Then the family \[\{E_s: y^2=x^3-1056x+13552+223s,s\in\Z\}\] having CM by $\mathcal{O}_K$ satisfies $|\overline{E_{s,223}}(\F_{223})|=223$.

\item  Let $K=\Q(\sqrt{-19})$ and $p=43$. Then the family $\{E_l:y^2=x^3-152x+722+43l,l\in\Z\}$ having CM by $\mathcal{O}_K$ satisfies $|\overline{E_{l,43}}(\F_{43})|=43$.

\item For $K=\Q(\sqrt{D})$ with $D=-43, -67, -163$ and a prime $p$ of the form $4p=1-Dv^2$ we can also describe infinite families having the desired cardinality of the special fiber. See \cite[Example 2.7]{Gazaki2022weak}). 
\end{itemize} For the remaining three quadratic imaginary fields of class number one, $\Q(i), \Q(\sqrt{-2})$ and $ \Q(\sqrt{-7})$, there are no examples that satisfy \autoref{Brorthogonal}.
However, for $\Q(i)$ the family $\{E_t: y^2=x^3+(3+5n):n\in\Z\}$ and the prime $p=5$ we have $|\overline{E_t}(\F_5)|=10$, which is divisible by $5$. All the above results can be extended to this case. This particular example has been considered in \cite[Theorem 4.9]{Wills25}. 

\end{exmp} 

\begin{rem} It follows by \cite[Proposition 4.4]{Gazaki2022weak}) that the extension $L=K(E[\eta])$ constructed in \autoref{Brorthogonal} is minimal in the following sense. If $\Q\subsetneq F\subsetneq L$ is an intermediate extension, then 
\[\prod_{w\in\Omega(F), w|p}T(A_w)_{\nd}\{p\}=0.\] Since we have a surjection $T(A_w)_{\nd}\{p\}\twoheadrightarrow A_0(X_w)_{\nd}\{p\}$, the same is true for the Kummer surface. This suggests that the smallest ramification index for which we might see involvement of the good reduction places in the Brauer-Manin set for zero-cycles of degree $0$ is $p-1$. A result of similar flavor was recently obtained for the Brauer-Manin set for rational points on $K3$ surfaces in \cite[Theorem 1.3]{pagano2023role} motivated by earlier considerations in \cite[Theorem C]{BrightNewton}. 
\end{rem}
\subsection*{Unconditional Exactness} We see that in the situation of \autoref{Brorthogonal} the complex \eqref{complexp} for the Kummer surface $X_L$ becomes 
\[A_0(X_L)/p\to A_0(X_{L_v})_{\nd}\{p\}=\Z/p\to 0.\] 
\autoref{localglobal1} (or equivalently \cite[Theorem 1.3]{Ieronymou2021}) allows us to lift the local zero-cycles to global, assuming that the Brauer-Manin obstruction is the only obstruction to Weak Approximation for rational points on every finite extension $F/L$. 
We close this article by describing cases when exactness can be proved unconditionally. 
\autoref{commutes} reduces the problem to proving exactness of the corresponding complex for the abelian surface, 
\[T(A_L)/p\to T(A_{L_v})_{\nd}\{p\}\to 0.\] 
In \cite[Theorem 1.7]{Gazaki2022weak} we gave a sufficient condition for elliptic curves with positive rank over $\Q$ to satisfy this unconditional exactness. 
We thus obtain the following Corollary.
\begin{cor}\label{reduction} Suppose we are in the set-up of \autoref{Brorthogonal}. Suppose additionally that the elliptic curve $E$ has positive rank and there exists a global point $P\in E(\Q)$ of infinite order satisfying the assumptions of \cite[Theorem 1.7]{Gazaki2022weak}. Then the complex \[A_0(X_L)/p\to A_0(X_{L_v})_{\nd}\{p\}=\Z/p\to 0\] is exact unconditionally on the arithmetic of rational points on the Kummer surface $X_L$. 
\end{cor} 

\begin{rem}
The assumption needed on the point $P\in E(\Q)$ is to satisfy a certain local behavior which we review here. In \cite[Section 3.2.1]{Gazaki2022weak} we constructed a local decomposition 
\[P_{\Q_p}=\widehat{P_{\Q_p}}\oplus \overline{P_{\Q_p}},\] where $P_{\Q_p}$ is the image of the point $P$ under the composition
\[E(\Q)\hookrightarrow E(\Q_p)\twoheadrightarrow E(\Q_p)/p,\]
$\widehat{P_{\Q_p}}$ is a ``formal point" in the formal group $\widehat{E_{\Q_p}}(\Z_p)/p$ and $\overline{P_{\Q_p}}$ is a point on the special fiber $\overline{E_p}(\F_p)=\overline{E_p}[p]$. The desired condition is for $\widehat{E_{\Q_p}}$ to be nontrivial. In \cite[Remark 4.8]{Gazaki2022weak}) we wrote an algorithm to check if such a global point exists. For small values of the prime $p$ this condition can be checked computationally. Such a computation was carried out in the Appendix of \cite{Gazaki2022weak} written by A. Koutsianas and it gave rise to the \autoref{weakevidence} mentioned in the introduction. Many more examples have been recently constructed in \cite{Wills25}, where there is also a heuristic argument that suggests that the sufficient condition is satisfied by a positive proportion of elliptic curves in each family.  
\end{rem}

\vspace{2pt}

\vspace{30pt}
		
		\bibliographystyle{amsalpha}
		
		\bibliography{bibfile,bibfileNSFCareer}
	
	\end{document}